\documentclass[a4paper,12pt]{article}
\usepackage{amsmath}
\usepackage{amssymb}
\usepackage{amscd}
\usepackage{amsthm}
\usepackage{enumerate}
\usepackage{amsfonts}
\usepackage{tikz}
\usetikzlibrary{cd}
\usepackage{microtype}
\usepackage{geometry}
\newtheorem{theorem}{Theorem}
\newtheorem{lemma}{Lemma}
\newtheorem{definition}{Definition}

\newtheorem{corollary}{Corollary}

\newcommand{\cV}{\mathcal V}

\newcommand{\cS}{\mathcal S}

\newcommand{\fF}{\mathfrak F}

\newcommand{\PG}{\mathrm{PG}\,}

\newcommand{\PGL}{{\mathrm{PGL}}\,}

\newcommand{\rank}{{\mathrm{rank}}\,}
\newcommand{\FF}{{\mathbb F}}

\newtheorem{proposition}[theorem]{Proposition}
\def\rif#1{(\ref{#1})}
\def\eqn#1$$#2$${\begin{equation}\label#1#2\end{equation}}

\def\F{\mathbb F}

\def\Z{\mathbb Z}

\def\cS{{\cal S}}

\def\cV{{\cal V}}

\def\la{\langle}
\def\ra{\rangle}
\DeclareMathOperator*{\opl}{\oplus}

\begin{document}
\title{On subspaces defining linear sets of maximum rank}
\author{Valentina Pepe}
\maketitle

\begin{abstract}
		
		Let $V$ denote an $r$-dimensional $\F_{q^n}$-vector space. Let $U$ and $W$ be $\F_q$-subspaces of $V$ and let $L_U$ and $L_W$ be the projective points of $\PG(V,q^n)$ defined by  $U$ and $W$ respectively. We address the problem of when $L_W=L_U$ under the hypothesis that $U$ and $W$ have maximum dimension, i.e.,  $\dim_{\F_q} W=\dim_{\F_q}U=$ $rn-n $, and we give a complete characterization for $r=2$.
	\end{abstract}

	\section{Introduction}

	Let $\F_{q}$ be the finite field of order $q$, where $q$ is a power of the prime $p$ and let $\F_q^*$ be $\F_q\setminus \{0\}$. Throughout the paper, by dimension we will mean the vectorial dimension.
	Let $V$ be an $r$-dimensional vector space over $\F_{q^n}$ and let $\PG(V,\F_{q^n})=\PG(r-1,q^n)$ be the associated projective space. We will often identify a projective point with one of its representative vectors.
	A point set $L$ of $\PG(V,\F_{q^n})$ is called an \emph{$\F_q$-linear set} of {\it rank} $m$ if it is
	defined by the non-zero vectors of a $m$-dimensional $\F_q$-vector subspace $U$ of $V$, i.e.
	\[L=L_U=\{\la {\bf u} \ra_{\F_{q^n}} \colon {\bf u}\in U\setminus \{{\bf 0} \}\}.\]
	
	By Grassmann's Identity, it can be easily seen that if $m>(r-1)n$ then $L_U=\PG(V,\F_{q^n})$, hence  $L_U$ is a proper subset of $\PG(V,\F_{q^n})$ if and only if its rank is at most $rn-n$. A linear set of rank $rn-n$ is said to be of \textit{maximum rank}.
	
	\medskip
	For a point $P=\la {\bf v} \ra_{\F_{q^n}}\in \PG(V,\F_{q^n})$ the \emph{weight} of $P$ with respect to the linear set $L_U$ is $w_{L_U}(P):=\dim_q(\la {\bf v} \ra_{\F_{q^n}} \cap U)$.
	
	\medskip
	
	The \emph{maximum field of linearity} of an $\F_q$-linear set $L$ is $\F_{q^d}$ if $d \mid n$ is the largest integer such that $L=L_U$ for some $\F_{q^d}$-subspace $U$, cf. \cite[pg.\ 403]{CsMP}.
	
	In  recent years, starting from the paper \cite{Lu1999} by Lunardon, linear sets have been used to construct or characterize various objects in finite geometry, such as blocking sets and multiple blocking sets in finite projective spaces, two-intersection sets in finite projective spaces, translation spreads of the Cayley Generalized Hexagon, translation ovoids of polar spaces, semifield flocks, finite semifields and rank metric codes. For a survey on linear sets, we refer the reader to \cite{OP2010}.
	
	\medskip


It is well known that different vector spaces - even with different dimensions - can define the same linear set.  For example, let $V=\F_{q^4}^2$, $U=\F_{q^2}^2$, $W=$ $\{(x,y), x \in \F_{q^2}, y \in \F_q\}$, then $\dim_{\F_q} U=4$, $\dim_{\F_q} W=3$ and $L_U=L_W \cong \PG(1,q^2)$ embedded in $\PG(1,q^4)$.

Also if we restrict ourselves to vector spaces of the same dimension, we can take  $W=\lambda U$ for some $\lambda \in \F_{q^n}^*$ and we easily get $L_U=L_W$.

In \cite{OP2010}, another way is shown  to get different vector spaces that determine the same linear set when $\dim_{\F_{q^n}} V=2$ and the linear sets have maximum rank. Let $Tr: x\in \F_{q^n}\mapsto x+x^q+\cdots + x^{q^{n-1}}\in \F_q$ the usual trace map, let $V=\{(x,y), x,y \in \F_{q^n}\}$ and let $\mathbf{b}$ the $\F_q$-bilinear skew-symmetric form defined on $V \times V$ such that $\mathbf{b}((x_1,y_1),(x_2,y_2))=Tr(x_1y_2-x_2y_1)$. A projective point $P=(a,b)$ of $\PG(V,q^n)=\PG(1,q^n)$ is an $n$-dimensional vector space $\Pi_P=\{\lambda (a,b), \lambda \in \F_{q^n}\}$ over $\F_q$ and $\Pi_P$ is a maximal totally isotropic space with respect to $\mathbf{b}$. Let $U$ an $\F_{q}$-vector subspace of $V$ of dimension $n$, i.e., $L_U$ has maximum rank and let $\perp$ be the polarity associated to $\mathbf{b}$. Then

$$\dim_{\F_q} \Pi_P \cap U \neq 0 \Leftrightarrow \dim_{\F_q} (\Pi_P + U^{\perp}) \neq 2n \Leftrightarrow \dim_{\F_q} \Pi_P \cap U^{\perp} \neq 0.$$

\noindent Therefore, $L_U=L_{U^{\perp}}$. Also, in the same way, it is easy to see that $w_{L_U}(P)=w_{L_{U^{\perp}}}(P)$.

In \cite{CsMP}, it is proved that for $n \leq 4$, $L_U=L_W$ if and only if $W=\lambda U$ or $W=\lambda U^{\perp}$ for some $\lambda \in \F_{q^n}^*$. Also, If $L_U$ is a club, then  $L_U=L_W$ if and only if $W=\lambda U$ for some $\lambda \in \F_{q^n}^*$. We recall that  a \emph{club} is a linear set $L_U$ of rank $n$ of PG$(1,q^n)$ with one point of weight $n-1$ and hence all the others with weight 1. In our general treatment, we will include a new proof of these two results with our methods.

Finally, in \cite{CsZ}, the authors  show that a particular class of linear sets of $\PG(1,q^n)$, the so called linear sets of \emph{pseudoregulus type}, can be determined by $\F_q$-vector spaces $U$ and $W$ of $V=\F_{q^n}^2$ of dimension $n$ such that $W \neq U^{\perp}$ and such that $U$ and $W$ do not belong to the same orbit under the action of the projective semilinear group $P\Gamma L(2,q^n)$, so $W \neq \lambda U$ for all $\lambda \in \F_{q^n}$. Let $N: x \in \F_{q^n} \mapsto x^{1+q+\cdots +q^{n-1}} \in \F_q$ be the usual norm map and let $L=\{(1,t), t \in \F_{q^n} : N(t)=1\}$. Then $L$ is a linear set of pseudoregulus type and $L=L_U=L_W$, with $U=\{(x,x^q), x \in \F_{q^n}\}$ and $W=\{(y,y^{q^i}), y \in \F_{q^n}\}$ such that $GCD(n,i)=1$ and $i \neq n-1$. So far, the linear sets of pseudoregulus type are the only known linear sets of $\PG(1,q^n)$ that can be determined by vector spaces $U$ and $W$ of the same dimension but not in the same $P\Gamma L(2,q^n)$-orbit and such that $U \neq \lambda W^{\perp}$. In general, a linear set of pseudoregulus type $L_U$ is such that $U=\{\lambda v_0+ \lambda^{q^i} v_1, \lambda \in \F_{q^n}\}$ for some vectors  $v_0,v_1 \in V$ such that $\dim_{\F_{q^n}} \la v_0,v_1\ra=2$ and the maximum field of linearity of $L_U$ is $\F_{q^d}$, with $d=GCD (n,i)$ (\cite{LMPT14}, \cite{CsZ2}).

From now on, let $d\mid n$ and let

$$Tr_{q^n|q^d}: x \in \F_{q^n}\mapsto x+x^{q^d}+\cdots + x^{q^{n-d}}\in \F_{q^d}$$

\noindent and

$$N_{q^n|q^d}:  x \in \F_{q^n}\mapsto x^{1+q^d+\cdots+q^{n-d}}\in \F_{q^d}$$

\noindent the usual trace and norm maps respectively.

We find  new classes of linear sets of $\PG(1,q^n)$, $n \geq 6$ not a prime, that can be determined by $\F_q$-vector spaces $U$ and $W$ of dimension $n$ such that $W \neq \lambda U^{\perp}$ and $W \neq \lambda U$ for any $\lambda \in \F_{q^n}^*$. Let $d$ be a non-trivial divisor of $n$, $g(y),g'(y)\in \F_{q^d}[y]$ be $\F_{q}$-linear maps and $f(x)$ an $\F_{q^d}$-linear map. Let $U'=\{(y,g(y)),y\in \F_{q^d}\}$ $W'=\{(y,g'(y)),y\in \F_{q^d}\}$,  $U=\{(x,f(x)+g(Tr_{q^n|q^d}(ax)),x\in \F_{q^n}\} $ and

\noindent $W=\{(x,f(x)+g'(Tr_{q^n|q^d}(ax)),x\in \F_{q^n}\}$. We will show in Theorem \ref{last} that if  $L_{U'}=L_{W'}$, then $L_U=L_{W}$, even if $W \neq \lambda U$ or $\lambda U^{\perp}$. If $L_{U'}=L_{W'}$ is a linear set of $\PG(1,q^d)$ of pseudoregulus type  we call $L_U=L_W$ a linear set of  \textit{generalized pseudoregulus type}. Let $\perp_d$ be the polarity induced by

\noindent $\mathbf{b}_d:((x_1,y_1),(x_2,y_2))\in \F_{q^d}^2\times \F_{q^d}^2 \mapsto Tr_{q^d|\F_q}(x_1y_2-x_2y_1) \in \F_q$.  If $W'=\lambda U'^{\perp_d}$, then we say that $W$ is the  \textit{generalized perp} of $U$.

\medskip

Our main results are the following.

\begin{theorem}
Let $V=\F_{q^n}^r$, $U$ and $W$ be two $(rn-n)$-dimensional vector spaces of $V$ such that $L_U=L_W$ and  let $\F_q$ be the maximum field of linearity for $L_U=L_W$. Then, for every point $P \in L=L_U=L_W$, we have that $w_{L_U}(P)=w_{L_W}(P)$. If $r=2$ and $\perp$ be the polarity associated to $\mathbf{b}:((x_1,y_1),(x_2,y_2))\in V\times V \mapsto Tr(x_1y_2-x_2y_1) \in \F_q$, then, one of these possibilities must occur:

  \begin{enumerate}
    \item $W=\lambda U$ for some $\lambda \in \F_{q^n}^*$;
    \item $W=\lambda U^{\perp}$ for some $\lambda \in \F_{q^n}^*$;
    \item $n\geq 5$ and $L_U=L_W$ is a linear set of pseudoregulus type;
    \item $n \geq 10$ is not a prime and  $L_U=L_W$ is a linear set of generalized pseudoregulus type;
    \item $n \geq 6$ is not a prime and $W$ is the generalized perp of $U$.
  \end{enumerate}
\end{theorem}

 \begin{theorem}
 Let $\dim_{\F_{q^n}}V >2$.
Let $U$ and $W$ be $\F_q$-vector spaces such that $\dim_{\F_q} U=\dim_{\F_q}W=(r-1)n$ and $L_U = L_W$ in $\PG(r-1,q^n)$. Let $U_0$ and $W_0$ be the intersection of $U$ and $W$ respectively with a line $\ell$ of $\PG(r-1,q^n)$, such that $\dim_{\F_q} U_0 = \dim_{\F_q} W_0 =n$ and $L_{U_0} = L_{W_0}$ is a linear set of pseudoregulus type with maximum field of linearity $\F_q$. Then either $W=\lambda U$ for some $\lambda \in \F_{q^n}^*$, or $L_U=L_W$ is cone with vertex a codimension 2 subspace of $\PG(r-1,q^n)$ disjoint from $\ell$ and base $L_{U_0}=L_{W_0}$.
\end{theorem}

We also give a few examples of $\F_q$-vector spaces $U$ and $W$ such that   $\dim_{\F_q} U=\dim_{\F_q}W$, $L_U = L_W$ in $\PG(r-1,q^n)$ and $W \neq \lambda U$ for any $\lambda \in \F_{q^n}^*$, but it is still an open problem whether these examples  are the only possible constructions with that property.

The paper is organized as follows. In Section 2 we will introduce the geometric and algebraic setting, in particular the so-called cyclic representation of PG$(r-1,q^n)$ with the associated Desarguesian spread, and how to use Dickson matrices to prove our results. In Section 3 we will characterize Dickson matrices having equal corresponding principal minors. In Section 4 we will present our main results.

\section{The geometric and algebraic setting}

Let us start by describing the geometric setting we adopt to study $\FF_q$--linear sets of $\PG(V,\F_{q^n})= \PG(r-1,q^n)$ (see \cite{Lu1999} and \cite{giuzzipepe15}).
	
	Regarding the $r$-dimensional $\FF_{q^n}$-vector space $V$ as an $\FF_q$--vector space of dimension $rn$, the points of $\PG(r-1,q^n)$ correspond to a partition of $\PG(rn-1,q)$ into $(n-1)$-dimensional projective subspaces. Such a partition, say $\mathcal S$, is called a \textit{Desarguesian spread} of $\PG(rn-1,q)$ and the pair $(\PG(rn-1,q),{\mathcal S})$ is said to be the $\FF_q$-\textit{linear representation} of $\PG(r-1,q^n)$. In this setting, an $\FF_q$-linear set $L_U$ of $\PG(r-1,q^n)$ is the subset of $\mathcal{S}$ consisting of the elements with non-empty intersection with the projective subspace $\PG(U,\FF_q)$ of $\PG(rn-1,q)$ defined by $U$.
	
	Consider the following \emph{cyclic representation} of $\PG(rn-1,q)$ in $\PG(rn-1,q^n)$. Let $\PG(rn-1,q^n)= \PG(V',\F_{q^n})$, where $V'=\F_{q^n}^{rn}$ and let ${\mathbf e}_i$ be the $i$-th canonical basis element of $V'$. Let $\sigma$ be the $\F_q$-semilinear map defined by the rule ${\mathbf e}_i\mapsto {\mathbf e}_{i+r}$, where the subscripts are taken mod $rn$ and with accompanying automorphism $x\in\F_{q^n}\mapsto x^q\in\F_{q^n}$. Then $\sigma$ has order $n$ and the $\FF_q$-vector subspace
$$\mathrm{Fix}\,\sigma=\{(\mathbf{x},\mathbf{x}^q,\ldots,\mathbf{x}^{q^{n-1}}),\mathbf{x}=x_0,x_1,\ldots,x_{r-1} \mbox{ and }x_i \in \FF_{q^n}\}$$
 of $V'$ defines a set of points of $\PG(rn-1,q^n)$  fixed by $\sigma$, i.e. a subgeometry $\Sigma=\PG(rn-1,q)$. The elements of $\mathcal{S}$ are the subspaces $\Pi_P:=\langle P,P^{\sigma},\ldots,P^{\sigma^{n-1}} \rangle \cap \Sigma$, with $P\in \Pi_0 \cong \PG(r-1,q^n)$ and $\Pi_0=\PG(V_0,\FF_{q^n})$, where

 \noindent $V_0:=\{({\mathbf x}, {\mathbf 0},\ldots,{\mathbf 0})\colon {\mathbf x}=x_0,\ldots,x_{r-1} \mbox{ and }x_i \in \FF_{q^n}\}$ (see \cite{Lu1999}). Let $\Pi_i=\PG(V_i,\FF_{q^n})$ be $\Pi_0^{\sigma^i}$, where $V_i=\langle \mathbf{e}_{h+ir}, h=0,1,\ldots,r-1\rangle$. In the following, we shall identify a point $P$ of $\Pi_0= \PG(r-1,q^n)$ with the spread element $\Pi_P$. We observe that $P$ is just the projection of $\Pi_P$ from $\langle \Pi_1,\Pi_{2},\ldots,\Pi_{n-1}\rangle$ on $\Pi_0$. If $L_U$ is a linear set of rank $m$, then it is induced by the projective subspace $\PG(U,\FF_q) \subset \Sigma$ and it can be viewed both as the subset of $\Pi_0$ that is the projection of $\PG(U,\F_q)$ from $\langle \Pi_1,\Pi_{2},\ldots,\Pi_{n-1}\rangle$ on $\Pi_0$, as well as the subset of $\mathcal{S}$ consisting of the elements $\Pi_P$ such that $\Pi_P \cap \PG(U,\FF_q) \neq \emptyset$.  We stress that we have defined the subspaces $\PG(U,\FF_q)$ and $\Pi_P$ as subspaces of $\Sigma=\PG(rn-1,q)$. In particular we will denote by $U(\F_{q^n})$ the span of $U$ over $\F_{q^n}$. Since $U(\F_{q^n})$ is spanned by vectors fixed by $\sigma$, $U(\F_{q^n})^{\sigma}=U(\F_{q^n})$ and hence $\dim_{\F_{q^n}} U(\F_{q^n})=\dim_{\F_q}U$ (see, e.g., \cite[Lemma 1]{Lu1999}). Analogously,

$$(\langle P,P^{\sigma},\ldots,P^{\sigma^{n-1}} \rangle \cap U(\F_{q^n}))^{\sigma}= \langle P,P^{\sigma},\ldots,P^{\sigma^{n-1}} \rangle \cap U(\F_{q^n})$$

so

\begin{equation}\label{dimension}
\dim_{\F_{q^n}}\langle P,P^{\sigma},\ldots,P^{\sigma^{n-1}} \rangle \cap U(\F_{q^n})=\dim_{\F_q} \Pi_P \cap U
\end{equation}

	Note that, if a point $P$ has weight $i$ in $L_U$ then \[\dim_{\F_q}\left( \Pi_P\cap U \right)= i.\]
	
	\medskip
	
	The image under the Grassmann embedding $\varepsilon_n$ of a Desarguesian spread $\cS$ of $\PG(rn-1,q)$ determines the algebraic variety $\cV_{rn} \subset \PG(r^n-1,q)$ (see \cite{giuzzipepe15}).
	In more details, put $P=\langle (x_0,x_1,\ldots,x_{r-1},0,0,\ldots,0)\rangle$, then the Grassmann embedding of $\Pi_P$ is the point defined by the vector of the minors of order $n$ of

	\[\left(
	\begin{array}{ccccccccccccc}
		x_0 & x_1 &\cdots & x_{r-1} & 0 & 0 & \cdots & \cdots & \cdots &0 & 0 &\cdots & 0  \\
		0 & 0 &\cdots  & 0 & x_0^q & x_1^q & \cdots & x_{r-1}^q& \cdots & \cdots &\cdots & \cdots  & \cdots\\
		\cdots & \cdots & \cdots & \cdots & \cdots & \cdots & \cdots & \cdots & \cdots & \cdots& \cdots & \cdots & \cdots\\
		0 & 0 & \cdots & 0 & 0 & 0 & \cdots & 0 & \cdots &  x_0^{q^{n-1}}& x_1^{q^{n-1}} &\cdots &  x_{r-1}^{q^{n-1}} \\
	\end{array}
	\right).\]
	
	If we disregard the identically zero minors, then the Grassmann embedding of the Desarguesian spread $\mathcal{S}$, i.e. $\cV_{rn}$,   is the image of the map
	\[ \alpha: (x_0,\ldots,x_{r-1})\in\PG(r-1,q^n)\mapsto
	 \left(\prod_{i=0}^{n-1} x_{f(i)}^{q^i}\right)_{f \in \fF}\in \PG(r^n-1,q)
	\subset \PG(r^n-1,q^n),\]
	where $\fF=\{ f: \{0,\ldots, n-1\}\to\{0,\ldots,r-1\} \}$.
	Here, $\alpha$ is the map that makes the following diagram commute:
	$$\begin{tikzcd}[row sep=large]
		\PG(r-1,q^n)\arrow[r, dotted, "\alpha"] \arrow[d ,"\text{$\FF_q$-linear repres.}"'] &\PG(r^n-1,q)  \\
		\mbox{\begin{minipage}{6cm}\begin{centering}
					$\PG(rn-1,q)$ \\ $\cS =$ {Desarguesian Spread} \end{centering}
		\end{minipage}}
		\arrow[ru, ->, "\varepsilon_n"'] & \\
	\end{tikzcd}
	.$$
	For more details see \cite{giuzzipepe15}.
	
	Let $\mathbf{x}^{(i)}:=(x_0^{(i)},x_1^{(i)},\ldots,x_{r-1}^{(i)})\in \mathbb{F}^r$, where $\mathbb{F}$ is any field. Then the Segre variety
	
	\[\displaystyle \Sigma_{rn}:=\underbrace{\PG(r-1,\mathbb{F})\otimes \PG(r-1,\mathbb{F})\otimes \cdots \otimes \PG(r-1,\mathbb{F})}_{n \text{ times}}\subset \PG(r^n-1,\mathbb{F})\]
	
	is the image of the map:
	
	\[s:\underbrace{\PG(r-1,\mathbb{F})\times \PG(r-1,\mathbb{F})\times \cdots \times \PG(r-1,\mathbb{F})}_{n \text{ times}}\longrightarrow \PG(r^n-1,\mathbb{F})\]

\[(\mathbf{x}^{(0)},\mathbf{x}^{(1)},\ldots,\mathbf{x}^{(n-1)}) \mapsto  \left(\prod_{i=0}^{n-1} x_{f(i)}^{(i)}\right)_{f \in \fF} .\]
	
	Let $\mathcal{T}$ be the collection of subspaces of $\PG(rn-1,q^n)$ of type $\la P_0, P_1,\ldots,P_{n-1}\ra$ with $P_i=\la v_i \ra_{\F_{q^n}} \in \Pi_i$, then $\Sigma_{rn}$ is the Grassmann embedding of $\mathcal{T}$. that is, for $v_i \in V_i \setminus \{\mathbf{0}\}$, by abuse of notation, we write $v_0\wedge v_1\wedge \cdots \wedge v_{n-1}=v_0\otimes v_1 \otimes \cdots \otimes v_{n-1}$.
	
	Hence we observe that $\cV_{rn}$ is the subvariety  of $\Sigma_{rn}$ fixed by the semilinear map \[\hat\sigma:\ v_0\otimes v_1 \otimes \cdots \otimes v_{n-1}\mapsto v_{n-1}^{\sigma}\otimes v_0^{\sigma} \otimes \cdots \otimes v_{n-2}^{\sigma},\] with $\la v_i \ra_{\F_{q^n}} \in \Pi_i \cong \PG(r-1,q^n)$.
	
	We recall the following result, which we will explain below.
	
	\begin{theorem}\cite{giuzzipepe15}
		The Grassmann embedding of a linear set $L_U$ of rank $m$ of $\PG(r-1,q^n)$ is the intersection of $\cV_{rn}$
		with a linear subspace.
		In particular, if the rank of $L_U$ is maximum, then the image of the linear set is a hyperplane section of $\cV_{rn}$.
	\end{theorem}

	Let $B$ be the $m\times rn$  matrix whose rows are an $\FF_q$-basis of $U$. It follows that the Grassmann embedding of the linear set $L_U$ is the algebraic variety that is also the linear section of $\mathcal{V}_{rn}$ defined by the condition that the rank of

	\[\left(
	\begin{array}{ccccccccccccc}
		&  &  &  &  &  & B &  &  &  &  &  &  \\
		x_0 & x_1 &\cdots & x_{r-1} & 0 & 0 & \cdots & \cdots & \cdots &0 & 0 &\cdots & 0  \\
		0 & 0 &\cdots  & 0 & x_0^q & x_1^q & \cdots & x_{r-1}^q& \cdots & 0 & 0 & \cdots  & 0 \\
		\cdots & \cdots & \cdots & \cdots & \cdots & \cdots & \cdots & \cdots & \cdots & \cdots& \cdots & \cdots & \cdots\\
		0 & 0 & \cdots & 0 & 0 & 0 & \cdots & 0 & \cdots &  x_0^{q^{n-1}}& x_1^{q^{n-1}} &\cdots &  x_{r-1}^{q^{n-1}} \\
	\end{array}
	\right)\]
	is less than $m+n$. In particular, if $L_U$ has maximum rank, that is $\dim U=rn-n$, then $L_U$ is the hyperplane section of $\cV_{rn}$ defined by

{\small \begin{equation}\det\left(
	\begin{array}{ccccccccccccc}
		&  &  &  &  &  & B &  &  &  &  &  &  \\
		x_0 & x_1 &\cdots & x_{r-1} & 0 & 0 & \cdots & \cdots & \cdots &0 & 0 &\cdots & 0  \\
		0 & 0 &\cdots  & 0 & x_0^q & x_1^q & \cdots & x_{r-1}^q& \cdots & 0 & 0 & \cdots  & 0 \\
		\cdots & \cdots & \cdots & \cdots & \cdots & \cdots & \cdots & \cdots & \cdots & \cdots& \cdots & \cdots & \cdots\\
		0 & 0 & \cdots & 0 & 0 & 0 & \cdots & 0 & \cdots &  x_0^{q^{n-1}}& x_1^{q^{n-1}} &\cdots &  x_{r-1}^{q^{n-1}} \\
	\end{array}
	\right)=0.\end{equation}}
	
	\bigskip

From now on, we will always assume that $\dim_{\F_q} U=rn-n$.

Let $X(U)$ be the determinantal variety linear section of $\cV_{rn}$ which is also the Grassmann embedding of the spread elements with non-empty intersection with $\PG(U,q)$. Hence, there is a one-to-one correspondence between the points of $L_U$ and $X(U)$ and by abuse of notation we will denote by the same letter the point of $L_U$ and its image in $X(U)$. We can also write $X(U)=V(F)$, where $F \in \F_{q^n}[x_0,x_1,\ldots,x_{n-1}]$ is a homogeneous polynomial of degree $1+q+\cdots + q^{n-1}$.

By $X(U(\F_{q^n}))$ we denote the hyperplane section of $\Sigma_{rn}$  which is the Grassmann embedding of the elements of $\mathcal{T}$ with non-empty intersection with $\PG(U(\F_{q^n}),q^n)$, hence $X(U(\F_{q^n}))=V(F(\mathbf{x}^{(0)},\mathbf{x}^{(1)},\ldots,\mathbf{x}^{(n-1)}))$, with $\mathbf{x}^{(i)}:=(x_0^{(i)},x_1^{(i)},\ldots,x_{r-1}^{(i)})\in \mathbb{F}_{q^n}^r$ and

\[F=\det\left(
	\begin{array}{ccccccccccccc}
		&  &  &  &  &  & B &  &  &  &  &  &  \\
		x_0^{(0)} & x_1^{(0)} &\cdots & x_{r-1}^{(0)} & 0 & 0 & \cdots & \cdots & \cdots &0 & 0 &\cdots & 0  \\
		0 & 0 &\cdots  & 0 & x_0^{(1)} & x_1^{(1)} & \cdots & x_{r-1}^{(0)}& \cdots & 0 & 0 & \cdots  & 0 \\
		\cdots & \cdots & \cdots & \cdots & \cdots & \cdots & \cdots & \cdots & \cdots & \cdots& \cdots & \cdots & \cdots\\
		0 & 0 & \cdots & 0 & 0 & 0 & \cdots & 0 & \cdots &  x_0^{(n-1)}& x_1^{(n-1)} &\cdots &  x_{r-1}^{(n-1)} \\
	\end{array}
	\right).\]
	
	\bigskip

\begin{theorem}\label{singular}
If $P$ is a point of weight $k\geq 2$ in $L_U$, then its image on $X(U)$ is a singular point of multiplicity $1+q+\cdots +q^{k-1}$.
\end{theorem}
\begin{proof}
Without loss of generality, we assume that $P=(1,0,0,\ldots,0)\in L_U$  has weight $k$.  We first study the multiplicity of $P$ as a point of $X(U)$ intersected with the embedding of a line $\ell$ through it. By Grassmann, $\dim U \cap \ell \geq n$. If $\ell \notin L_U$, then $\dim U \cap \ell = n$. So let $\ell$ be a line through $P$ not contained in $L_U$ and let $U_1$ be $U\cap \ell$. Then $L_{U_1}^{\alpha}$ is a hyperplane section of $\ell^{\alpha} \cong \cV_{2n}$. By the action of PGL$(r,q^n)$, we can assume that $\ell$ is defined by $x_i=0, i=2,3,\ldots,n-1$, and we can disregard $x_i$, $i>1$.  We observe that he weight of $P$ is $k$ also in $U_1$ and there exists  an $\F_{q}$-vector subspace $K$ of $\F_{q^n}$ of dimension $k$ such that $\Pi_P \cap U_1=\{\lambda(1,0), \lambda \in K\}$. Let $\{a_1,a_2,\ldots,a_k\}$ be an $\F_q$-basis for $K$. Then we can take as $\F_q$-basis of $U_1$ the vectors $\{(a_i,0),(a_j,b_j),i=1,2,\ldots,k,j=k+1,k+2,\ldots,n\}$. Then, $\{b_{k+1},b_{k+2},\ldots,b_n\}$ is also an independent set over $\F_q$. Hence, by \cite[Lemma 3.51]{finitefields}, the matrices

$$
M_1=\left(\begin{array}{cccc}
         a_1 & a_1^{q} & \cdots & a_1^{q^{k-1}} \\
         a_2 & a_2^{q} & \cdots & a_2^{q^{k-1}} \\
         \cdots & \cdots& \cdots & \cdots \\
         a_k & a_k^{q} & \cdots & a_k^{q^{k-1}} \\
       \end{array}\right)
$$
and

 $$
M_2=\left(\begin{array}{cccc}
        b_{k+1} & b_{k+1}^{q} & \cdots & b_{k+1}^{q^{n-k-1}} \\
        b_{k+2} & b_{k+2}^q  & \cdots & b_{k+2}^{q^{n-k-1}} \\
         \cdots & \cdots& \cdots & \cdots \\
         b_n & b_n^{q} & \cdots & b_n^{q^{n-k-1}} \\
       \end{array}\right)
$$

\noindent are non-singular.

In the cyclic representation, we have

\noindent$U_1=\la (a_1,0,a_1^q,0,\ldots,a_1^{q^{n-1}},0), (a_2,0,a_2^q,0,\ldots,a_2^{q^{n-1}},0),\ldots,(a_k,0,a_k^q,0,\ldots,a_k^{q^{n-1}},0), $ $ (a_{k+1},b_{k+1},a_{k+1}^q,b_{k+1}^q,\ldots,a_{k+1}^{q^{n-1}},b_{k+1}^{q^{n-1}}),\ldots,(a_n,b_n,a_n^q,b_n^q,\ldots,a_n^{q^{n-1}},b_n^{q^{n-1}})  \ra_{\F_q} $,

\noindent $X(U_1)=V(F(x,y))$ and

$$F(x,y)=\det\left(
  \begin{array}{cccccccc}
    a_1 & 0 & a_1^q & 0 & \cdots & \cdots& a_1^{q^{n-1}} &  0 \\
    a_2 &0 & a_2^q & 0 & \cdots & \cdots& a_2^{q^{n-1}} &  0 \\
    \cdots & \cdots & \cdots & \cdots & \cdots & \cdots & \cdots & \cdots \\
     a_k &0 & a_k^q & 0 & \cdots & \cdots& a_k^{q^{n-1}} &  0 \\
    a_{k+1} & b_{k+1} & a_{k+1}^q & b_{k+1}^q & \cdots & \cdots& a_{k+1}^{q^{n-1}} &  b_{k+1}^{q^{n-1}} \\
    \cdots & \cdots & \cdots & \cdots & \cdots & \cdots & \cdots & \cdots \\
     a_{n} & b_{n} & a_{n}^q & b_{n}^q & \cdots & \cdots& a_{n}^{q^{n-1}} &  b_{n}^{q^{n-1}} \\
    x & y & 0 & 0 & \cdots & \cdots & 0 & 0 \\
    0 & 0 & x^q & y^q & \cdots& \cdots & 0 & 0 \\
    \cdots & \cdots & \cdots & \cdots & \cdots & \cdots & \cdots & \cdots \\
    0& 0 & 0 & 0 & \cdots & \cdots & x^{q^{n-1}}& y^{q^{n-1}}\\
  \end{array}
\right).$$

\noindent The monomials appearing in $F(x,y)$ are of type $\displaystyle \prod_{i \in T} x^{q^i}\prod_{j \notin T}y^{q^j}$, where $T\subseteq \{0,1,2,\ldots,n-1\}$ with coefficient equals (up to the sign) to the determinant of the $n\times n$ submatrix of

$$\left(
  \begin{array}{cccccccc}
    a_1 & 0 & a_1^q & 0 & \cdots & \cdots& a_1^{q^{n-1}} &  0 \\
    a_2 &0 & a_2^q & 0 & \cdots & \cdots& a_2^{q^{n-1}} &  0 \\
    \cdots & \cdots & \cdots & \cdots & \cdots & \cdots & \cdots & \cdots \\
     a_k &0 & a_k^q & 0 & \cdots & \cdots& a_k^{q^{n-1}} &  0 \\
    a_{k+1} & b_{k+1} & a_{k+1}^q & b_{k+1}^q & \cdots & \cdots& a_{k+1}^{q^{n-1}} &  b_{k+1}^{q^{n-1}} \\
    \cdots & \cdots & \cdots & \cdots & \cdots & \cdots & \cdots & \cdots \\
     a_{n} & b_{n} & a_{n}^q & b_{n}^q & \cdots & \cdots& a_{n}^{q^{n-1}} &  b_{n}^{q^{n-1}} \end{array} \right)$$

\noindent containing $b_h^{q^i},a_h^{q^j}$ with $i \in T$ and $j \notin T$. If $|T|>n-k$, then  such a determinant is 0. Hence, the monomial with the minimum degree for $y$ in $F(x,y)$ is obtained by taking $T=\{k,k+1,\ldots, n-1\}$, so that its coefficient of $\displaystyle \prod_{i \in T} x^{q^i}\prod_{j \notin T}y^{q^j}$, up to the sign, is $\det N$, where

$$
N=\left(
  \begin{array}{cccccccc}
    a_1 & a_1^q &  \cdots & a_1^{q^{k-1}}& 0 & 0 &\cdots&   0 \\
    a_2 & a_2^q & \cdots &  a_2^{q^{k-1}} & 0& 0 &\cdots &  0 \\
    \cdots & \cdots & \cdots & \cdots & \cdots & \cdots & \cdots & \cdots \\
     a_k & a_k^q & \cdots &a_k^{q^{k-1}} & 0& 0 & \cdots  &  0 \\
    a_{k+1} &  a_{k+1}^q & \cdots & a_{k+1}^{q^{k-1}} & b_{k+1}^{q^k} & b_{k+1}^{q^{k+1}} &\cdots &  b_{k+1}^{q^{n-1}} \\
     a_{k+2} &  a_{k+2}^q & \cdots & a_{k+2}^{q^{k-1}} & b_{k+2}^{q^k} & b_{k+2}^{q^{k+1}} & \cdots &  b_{k+2}^{q^{n-1}} \\
    \cdots & \cdots & \cdots & \cdots & \cdots & \cdots & \cdots & \cdots \\
     a_{n} &a_{n}^q & \cdots &a_{n}^{q^{k-1}}& b_n^{q^k} & b_n^{q^{k+1}} &  \cdots &  b_{n}^{q^{n-1}} \end{array} \right).
$$

We have $\det N=\det M_1 \det M_2^{q^k}\neq 0$.

Therefore, the minimum degree for $y$ is exactly  $1+q+\cdots q^{k-1}$ and $(1,0)$ has multiplicity  $1+q+\cdots + q^{k-1}$. Therefore, by the action of the lifting of PGL$(r,q^n)$ on $\cV_{rn}$, the multiplicity of $P$  on the embedding of any line through $P$ not contained in $L_U$ is $1+q+\cdots + q^{k-1}$. To determine the multiplicity of $P$ as a point of $X(U)=V(F)$, we consider the affine cover defined by $x_0 \neq 0$. We get

$$f(x_1,x_2,\ldots,x_{r-1}):=F(1,x_1,x_2,\ldots,x_{r-1})= \displaystyle\sum_{j=m}^{1+q+\cdots + q^{n-1}}f_j(x_1,x_2,\ldots,x_{r-1}), $$ where $f_j(x_1,x_2,\ldots,x_{r-1})$ is homogeneous of degree $j$. By the multiplicity of $P$ on the embedding of any line through it, we get that

$$f(\mu a_1,\mu a_2,\ldots, \mu a_{r-1})=  \displaystyle\sum_{j=1+q+\cdots q^{k-1}}^{1+q+\cdots + q^{n-1}}\mu^jf_j(a_1,a_2,\ldots,a_{r-1})$$ with $f_{1+q+\cdots q^{k-1}}(a_1,a_2,\ldots,a_{r-1})\neq 0$, for some  $(a_1,a_2,\ldots,a_{r-1}) \neq \mathbf{0}$. Suppose that $m < 1+q+ \cdots q^{k-1}$, then $f_m(a_1,a_2,\ldots,a_{r-1})=\mathbf{0}$ with

\noindent $\deg f_m < 1+q+\cdots q^{k-1} \leq 1+q+\cdots q^{n-1} < q^n$ for any prime power $q$, a contradiction. Hence, $m=1+q+\cdots q^{k-1}$ is the multiplicity of $P$ on $X(U)$.
\end{proof}

\begin{corollary}\label{rational}
Let $r=2$, that is $L_U$ is a linear set of maximum rank of PG$(1,q^n)$ and let $L_U=V(F(x,y))$, where $F(x,y)$ is defined by determinantal condition $(2)$. Then $V(F(x,y))$ has $\deg F = 1+q+ \cdots + q^{n-1}$ roots in $\F_{q^n}$ counted with their multiplicity.
\end{corollary}
\begin{proof}
  We have that $V(F(x,y))=|L_{U}|=\displaystyle \sum_{i=1}^{n}t_j$ and $\displaystyle \sum_{j=1}^{n}t_j\frac{q^j-1}{q-1}=1+q+\cdots + q^{n-1}$, where $t_j$ is the number of points of $L_{U_i}$ of weight $j$ (see \cite[Prop. 2.2]{OP2010}). Hence, by Proposition \ref{singular}, the number of zeros of $F(x,y)$ in $\PG(1,q^n)$ counted with their multiplicity is equal to the degree of $F(x,y)$, so all the zeros of $F(x,y)$ are in $\PG(1,q^n)$.
\end{proof}

We can prove the following Theorem, which is  part of our main results.

\begin{theorem}\label{main1}
Let $U$ and $W$ two $\F_q$-vector space of dimension $rn-n$ such that $L_U=L_W$. Let $X(U)=V(F)$, $X(W)=V(G)$, with $F,G \in \F_{q^n}[x_0,x_1,\ldots,x_{r-1}]$ defined by determinantal condition $(2)$. Then $G=\lambda F$ for some $\lambda \in \F_{q^n}^*$. Also, $w_{L_U}(P)=w_{L_W}(P)$ $\forall \, P \in L_U=L_W$.
\end{theorem}
\begin{proof}
We have that $V(F)=V(G)$ in PG$(r-1,q^n)$, that is $F$ and $G$ have the same set of $\F_{q^n}$-rational zeros. Suppose that there is a non-constant  polynomial $H(x_0,x_1,\ldots,x_{n-1})$ dividing $F$ such that $V(H)= \emptyset$ in PG$(r-1,q^n)$, that is $H$ does not have $\F_{q^n}$-rational zeros. Hence $F=HF'$ and $V(F)=V(F')$ with $\deg F'< \deg F $. By Corollary \ref{rational}, for a line $\ell$ not contained in $V(F)$, the number of points of $V(F)\cap \ell$ is $1+q+\cdots + q^{n-1}= \deg F$, a contradiction. The same is valid for $G$ of course. Hence, $F$ and $G$ are polynomials of minimal degree defining $X(U)$ and $X(W)$. Then, by Nullstellensatz over finite field, we have $V(F)=V(G)$ implies that $F \equiv \lambda G$ for some $\lambda \in \F_{q^n}^*$ in $\F_{q^n}[x_0,x_1,\ldots,x_{r-1}]/(x_0^{q^n}-x_0, x_1^{q^n}-x_1,\ldots,x_{r-1}^{q^n}-x_{r-1})$. Since $\deg F=\deg G =1+q+\cdots + q^{n-1} < q^{n}$ for any prime power $q$, we must have $G=\lambda F$. Also, since the singularity of the points is the same, by Theorem \ref{singular} so is the weight distribution in $L_U$ and $L_W$.
\end{proof}

We will need also the following result.

\begin{theorem}\label{dlinear}
 Let $U$ and $W$ two $\F_q$-vector space of dimension $rn-n$ such that $L_U=L_W$. If there exists $1 < d|n$ and  $u\neq \mathbf{0}$ such that $\la u \ra_{\F_{q^d}} \subset U$, then $\la \lambda u \ra_{\F_{q^d}} \subset W$ for some $\lambda \in \F_{q^n}$.
\end{theorem}

\begin{proof}
As in Theorem \ref{singular}, we can assume $r=2$ and,  up the action of $\PGL(2,q^n)$, that $u=(1,0)$. By Theorem \ref{main1},  $w_{L_U}(P)=w_{L_W}(P) \geq d$. In the cyclic representation, we have

\noindent$U=\la (1,0,1,0,\ldots,1,0), (\lambda_1,0,\lambda_1^q,0,\ldots,\lambda_1^{q^{n-1}},0),\ldots,(\lambda_{d-1},0,\lambda_{d-1}^q,0,\ldots,\lambda_{d-1}^{q^{n-1}},0), $  $(a_{1},b_{1},a_{1}^q,b_{1}^q,\ldots,a_{1}^{q^{n-1}},b_{1}^{q^{n-1}}),\ldots,(a_{n-d},b_{n-d},a_{n-d}^q,b_{n-d}^q,\ldots,a_{n-d}^{q^{n-1}},b_{n-d}^{q^{n-1}})  \ra_{\F_q} $,

\noindent with $\{1,\lambda_1,\ldots,\lambda_{d-1}\}$ an $\F_{q}$-basis for $\F_{q^d}$. Let

$$B_1=\left(
	\begin{array}{ccccccccc}
	1	& 0 & 1 & 0 & \cdots & \cdots & \cdots & 1 & 0 \\
\lambda_1	& 0 & \lambda_1^q & 0 & \cdots & \cdots & \cdots & \lambda_1^{q^{n-1}} & 0 \\
\cdots & \cdots & \cdots & \cdots & \cdots & \cdots & \cdots & \cdots & \cdots \\
\lambda_{d-1}	& 0 & \lambda_{d-1}^q & 0 & \cdots & \cdots & \cdots & \lambda_{d-1}^{q^{n-1}} & 0 \\
\end{array}
	\right)$$

\noindent and

$$B_2=\left(
	\begin{array}{ccccccccc}
a_{1} & b_{1} & a_{1}^q & b_{1}^q & \cdots & \cdots & \cdots &a_{1}^{q^{n-1}} & b_{1}^{q^{n-1}}\\
a_{2} & b_{2} & a_{2}^q & b_{2}^q & \cdots & \cdots & \cdots &a_{2}^{q^{n-1}} & b_{2}^{q^{n-1}}\\
\cdots & \cdots & \cdots & \cdots & \cdots & \cdots & \cdots & \cdots & \cdots \\
a_{n-d} & b_{n-d} &a_{n-d}^q & b_{n-d}^q &\cdots & \cdots & \cdots & a_{n-d}^{q^{n-1}} & b_{n-d}^{q^{n-1}}\\
\end{array}
	\right)
$$

and let $B_i[J]$ be the submatrix of $B_i$ obtained by taking the $j$-th column of $B_i$ for every $j \in J \subset \{0,1,\ldots,2n-1\}$.

 We want to show that $V(y_i, i \equiv 0 \mod d) \subset X(U(\F_{q^n}))$. For $y_i=0$ $\forall \, i \equiv 0 \mod d$, we have

\[\det\left(
	\begin{array}{cccccccc}
	&  & &  B_1 &  &  &    &  \\
	&  & &  B_2  & &  &    &  \\
a_{1} & b_{1} & a_{1}^q & b_{1}^q & \cdots  & \cdots & a_{1}^{q^{n-1}} & b_{1}^{q^{n-1}}\\
\cdots & \cdots & \cdots & \cdots & \cdots  & \cdots & \cdots & \cdots \\
a_{n-d} & b_{n-d} &a_{n-d}^q & b_{n-d}^q  & \cdots & \cdots & a_{n-d}^{q^{n-1}} & b_{n-d}^{q^{n-1}}\\
		x^{(0)} & y^{(0)}  & 0 & 0  & \cdots & \cdots &0 & 0   \\
		0 & 0 & x^{(1)} & y^{(1)}  & \cdots & \cdots & 0 & 0 \\
		\cdots & \cdots & \cdots  & \cdots & \cdots & \cdots & \cdots & \cdots \\
		0 & 0 & 0 & 0 &  \cdots  & \cdots & x^{(n-1)} &  y^{(n-1)} \\
	\end{array}
	\right)=\]
\medskip

\noindent $\displaystyle\prod_{i \equiv 0 \mod d}
\det \left(
	\begin{array}{cccccccccccc}
& & & & & B_1' & & & & & \\
 & & & & & B_2'& & & & &\\
 0 & x^{(1)} & y^{(1)} & \cdots & \cdots  & \cdots & \cdots & \cdots & \cdots & \cdots & 0 & 0   \\
 \cdots & \cdots & \cdots  & \cdots & \cdots  & \cdots & \cdots & \cdots  & \cdots & \cdots   & \cdots  & \cdots\\
 0 & 0 & 0 & \cdots & x^{(d-1)} & y^{(d-1)} & \cdots & \cdots & \cdots &\cdots &   0 & 0 \\
 0 & 0 & 0 & \cdots &  \cdots & \cdots & 0 & \cdots  & \cdots   &\cdots & 0 & 0 \\
  0 & 0 & 0 & \cdots &  \cdots & \cdots &\cdots & x^{(d+1)} & y^{(d+1)}  &\cdots & 0 & 0 \\
   \cdots & \cdots & \cdots & \cdots &\cdots  & \cdots & \cdots & \cdots  & \cdots & \cdots \\
   0 & 0 & 0 & \cdots &  \cdots &  \cdots &  \cdots  & \cdots & \cdots & \cdots & x^{(n-1)} & y^{(n-1)} \\
 \end{array}
\right)$,
\medskip

\noindent where $B_i'=B_i[2j \colon j \not \equiv 0 \mod d; 2j'+1 \colon j'=0,1,\ldots,n-1]$, $i=1,2$.
For every $j,j'$ with $j \equiv j'$ $\mod d$, we have $B_1[2j]=B_1[2j']$, and $B_1[2j + 1]=\mathbf{0}$ for all $j=0,1,\ldots,n-1$. Hence $\rank  B_1[2j \colon j \not \equiv 0 \mod d; 2j'+1 \colon j'=0,1,\ldots,n-1]=d-1$ and
\medskip
\[\det\left(
	\begin{array}{ccccccccccc}
  & & & & & B_1' & & & & & \\
& & & &  & B_2'& & & & &\\
 0 & x^{(1)} & y^{(1)} & \cdots & \cdots & \cdots & \cdots & \cdots & \cdots& 0 & 0   \\
 \cdots & \cdots & \cdots & \cdots & \cdots  & \cdots & \cdots & \cdots & \cdots \\
 0 & 0 & 0 & \cdots & x^{(d-1)} & y^{(d-1)} & \cdots & \cdots & \cdots &  0 & 0 \\
  0 & 0 & 0 & \cdots &  \cdots & \cdots & x^{(d+1)} & y^{(d+1)}  &\cdots & 0 & 0 \\
   \cdots & \cdots & \cdots & \cdots & \cdots & \cdots \\
   0 & 0 & 0 & \cdots &  \cdots &  \cdots & \cdots & \cdots & \cdots & x^{(n-1)} & y^{(n-1)} \\
 \end{array}
\right)\]
\medskip

is identically zero. By $X(U)=X(W)$, we get $X(U(\F_{q^n}))=X(W(\F_{q^n}))$. Let $\left(
                                                                                    \begin{array}{c}
                                                                                      C_1 \\
                                                                                      C_2 \\
                                                                                    \end{array}
                                                                                  \right)
$
be a matrix whose rows are a basis for $W$ in its cyclic representation, with the rows of $C_1$ being a basis for a $d$-dimensional subspace of $\la (1,0)\ra_{\F_{q^n}}\cap W$. As before,  since $\rank \left(\begin{array}{c}
C_1 \\
 C_2 \\
 \end{array}
  \right)=n$ and $C_1[2j+1]=\mathbf{0}$ for all $j=0,1,\ldots,n-1$, we must have $\rank C_2[2j \colon j=0,1,\ldots,n-1]=n-d$.
In order to have $X(U(\F_{q^n}))=X(W(\F_{q^n}))$, we must have

\[
\rank \left(
        \begin{array}{c}
           C_1[2j \colon j \not \equiv 0 \mod d; 2j'+1 \colon j'=0,1,\ldots,n-1] \\
        \end{array}
      \right) < d
\]

\noindent and since $C_1[2j'+1 \colon j'=0,1,\ldots,n-1]=\mathbf{0}$ for all $j'=0,1,\ldots,n-1$, we must have $\rank C_1[2j \colon j \not \equiv 0 \mod d] < d$. We observe that  $C_1[2j \colon j \not \equiv 0 \mod d]= C_1[2j \colon j \not \equiv d-1 \mod d]^q$, so $\rank C_1[2j \colon j \not \equiv d-1 \mod d] < d $. By \cite[Lemma 3.51]{finitefields}, $\rank C_1[2j, j=0,1,\ldots,d-1]=d$. Therefore,

\noindent $d-1=\rank  C_1[2j, j=0,1,\ldots,d-2]=\rank C_1[2j \colon j \not \equiv d-1 \mod d]$. Let $\mathbf{c}$ be $C_1[0]$, then $C_1[2j]=\mathbf{c}^{q^j}$, $j=1,2,\ldots,n-1$. Hence, $\mathbf{c}^{q^j} \in \la \mathbf{c},\mathbf{c}^q,\ldots, \mathbf{c}^{q^{d-2}} \ra$ for all $ \not \equiv d-1 \mod d$. Let $\mathbf{c}^{q^d}=\displaystyle\sum_{j=0}^{d-2}h_j\mathbf{c}^{q^j}$. We show by induction on $j$ that $h_{d-2-j}=0$ for all $j=0,1,2,\ldots, d-3$. By $\mathbf{c}^{q^d}=\displaystyle\sum_{j=0}^{d-2}h_j\mathbf{c}^{q^j}$, we get $\mathbf{c}^{q^{d+1}}=\displaystyle\sum_{j=0}^{d-2}h_j^q\mathbf{c}^{q^{j+1}}$, but we also have $\mathbf{c}^{q^{d+1}}=\displaystyle\sum_{i=0}^{d-2}k_i\mathbf{c}^{q^i}$, so, by the independence of $\{\mathbf{c},\mathbf{c}^q,\ldots,\mathbf{c}^{q^{d-1}}\}$, we get $h_{d-2}=0$. Suppose that $h_{d-2-j}=0$ for all $j$ such that $0 \leq j < \overline{j}$, then $\mathbf{c}^{q^{d+\overline{j}}}=\displaystyle\sum_{i=0}^{d-1-\overline{j}}h_i^{q^{\overline{j}}}\mathbf{c}^{i+\overline{j}}$. But we also have $\mathbf{c}^{q^{d+\overline{j}}} \in \la \mathbf{c},\mathbf{c}^q,\ldots, \mathbf{c}^{d-2} \ra $, hence $h_{d-2-\overline{j}}=0$. Hence, $\mathbf{c}^{q^d} = h \mathbf{c}$ for some $h \in \F_{q^n}^*$. So $\mathbf{c}=\lambda\left(
                                                            \begin{array}{c}
                                                              1 \\
                                                              \mu_1 \\
                                                              \vdots \\
                                                              \mu_{d-1} \\
                                                            \end{array}
                                                          \right)
$ for some  $\lambda \in \F_{q^n}^*$ and $\mu_i \in \F_{q^d}^*$, $i=1,2,\ldots, d-1$.

\end{proof}

Up to the action of $\PGL(r,q^n)$, we can assume that $P_{\infty}=(0,0,\ldots,1)\notin L_U$, hence we can write $U=\{(x_0,x_1,\ldots,x_{r-2}, \displaystyle\sum_{i=0}^{r-2}f_i(x_i)), x_i \in \F_{q^n}\}$, where $f_i\in \F_{q^n}[x_i]$ is a \emph{linearized polynomial} $\forall \, i=0,1,\ldots,r-2$.  A polynomial $f(x) \in \F_{q^n}[x]$ is said to be a linearized polynomial if $f(x)=a_0x+a_1x^{q}+\cdots +a_{n-1}x^{q^{n-1}}$, hence  $f$ is an endomorphism of $\F_{q^n}$ when regarded as $\F_q$-vector space. Since $\dim_{\F_q} End(\F_{q^n})=n^2$, by comparing dimension we easily get that $f$ is an endomorphism of $\F_{q^n}$ if and only if $f$ is a linearized polynomial. To a linearized polynomial $f(x)=a_0x+a_1x^{q}+\cdots +a_{n-1}x^{q^{n-1}}$  we can associate a \emph{Dickson matrix}

$$
A=\left(
    \begin{array}{cccc}
      a_0 & a_1 & \cdots & a_{n-1} \\
      a_{n-1}^q & a_0^q &  \cdots & a_{n-2}^q \\
       \cdots &  \cdots & \cdots &  \cdots \\
      a_1^{q^{n-1}} & a_2^{q^{n-1}} &  \cdots & a_0^{q^{n-1}} \\
    \end{array}
  \right)
$$

\noindent that is $A$ is a $n \times n$ matrix such that its $(ij)$-entry is $a_{j-i}^{q^i}$. The rank of $f$ seen as an endomorphism is the same as the rank of the associated Dickson matrix (see \cite{WuLiu}).

Hence, there is a one-to-one correspondence between the $\F_q$-vector subspace $U$ of $\F_{q^n}^r$ disjoint from the projective point $P_{\infty}$ and the ordered sets $\{A_0,A_1,\ldots,A_{r-2}\}$ of Dickson matrices of order $n$. We will say that $\{A_0,A_1,\ldots,A_{r-2}\}$ is \textit{associated} to $U.$

We can now prove a theorem that highlights the link between matrices associated to vector spaces determining the same linear sets.

\begin{theorem}\label{main2}
Let $U$ and $W$ be two $\F_q$-vector subspaces of dimension $rn-n$ such that $L_U$ and $L_W$ do not contain $P_{\infty}=(0,0,\ldots,1)$. Let $\{A_0,A_1,\ldots,A_{r-2}\}$ and $\{B_0,B_1,\ldots,B_{r-2}\}$  be the ordered sets of Dickson matrices associated to $U$ and $W$ respectively. Then $L_U=L_W$ if and only if for every $I\subseteq \{0,1,\ldots,n-1\}$ and for every  function $\psi:I\to\{0,\ldots,r-2\} $, $\det C_{\psi}=\det C'_{\psi}$, where $C_{\psi}$ and $C'_{\psi}$ are square matrices of order $|I|$ such that the component $ij$ of $C_{\psi}$ is the component $ij$ of $A_{\psi(j)}$ and the component $ij$ of $C'_{\psi}$ is the component $ij$ of $B_{\psi(j)}$, with $i,j \in I$.
\end{theorem}
\begin{proof}
Let $\mathbf{x}_i:=(x_i,x_i^q,\ldots,x_i^{q^{n-1}})$. After rearranging the components of $V'=\F_{q^n}^{rn}$, the cyclic representation of $U$ is

 $$U=\{(\mathbf{x}_0,\mathbf{x}_1,\ldots,\mathbf{x}_{r-2},\mathbf{x}_0A_0^{T},\mathbf{x}_1A_1^{T}, \ldots,\mathbf{x}_{r-2}A_{r-2}^{T}), x_i \in \F_{q^n},i=0,1,\ldots,r-2\}.$$

By \rif{dimension}, $\Pi_P \cap U \neq \mathbf{0}$ if and only if $\langle P,P^{\sigma},\ldots,P^{\sigma^{n-1}} \rangle \cap U(\F_{q^n}) \neq \mathbf{0}$, then we can write $X(U)=V(F(x_0,x_1,\ldots,x_{r-1}))$, with
 $F(x_0,x_1,\ldots,x_{r-1})=\det M$ and $M$ a matrix such that the first $rn-n$ rows form \emph{any} basis of $U(\F_{q^n})$ set-wised fixed by $\sigma$.

 We have that

 $$U(\F_{q^n})=\{(\mathbf{x}^{(0)},\mathbf{x}^{(1)},\ldots,\mathbf{x}^{(r-2)},\mathbf{x}^{(0)}A_0^{T},\mathbf{x}^{(1)}A_1^{T}, \ldots,\mathbf{x}^{(r-2)}A_{r-2}^{T}), x_i^{(j)} \in \F_{q^n},i=0,1,\ldots,r-2\},$$

\noindent hence we can take

\medskip
\noindent $
 M=\left(
     \begin{array}{ccccccccccccc}
        &  &  &  &  &  &  &  &  &  & A_0^T &  &  \\
        &  &  &  &  &  &  &  &  & & A_1^T &  &  \\
         &   &   &   &   & I_{rn-n}  &   &   &   &  \cdots &  \cdots &  \cdots &  \cdots \\
        &  &  &  &  &  &  &  &  &  & A_{r-1}^T &  &  \\
       x_0 & 0&  \cdots & 0 & x_1 & 0 &  \cdots & 0 &  \cdots & x_{r-1} & 0&  \cdots & 0  \\
       0 & x_0^q &  \cdots & 0 & 0 & x_1^q &  \cdots & 0 &  \cdots & 0 &x_{r-1}  & \cdots & 0 \\
        \cdots &  \cdots &  \cdots &  \cdots &  \cdots &  \cdots &  \cdots &  \cdots &  \cdots &  \cdots &  \cdots &  \cdots &  \cdots \\
       0 & 0 &  \cdots & x_0^{q^{n-1}} & 0 & 0 &  \cdots & x_1^{q^{n-1}} & \cdots & 0 & 0 & \cdots & x_{r-1}^{q^{n-1}} \\
     \end{array}
   \right).
 $

 \bigskip

For any $\phi:\{0,\ldots, n-1\}\to\{0,\ldots,r-1\} $, the coefficient of $\prod_{i=0}^{n-1} x_{\phi(i)}^{q^i}$ is the minor of

$$\left(
     \begin{array}{ccccccccccccc}
        &  &  &  &  &  &  &  &  &  & A_0^T &  &  \\
        &  &  &  &  &  &  &  &  & & A_1^T &  &  \\
         &   &   &   &   & I_{rn-n}  &   &   &   &  \cdots &  \cdots &  \cdots &  \cdots \\
        &  &  &  &  &  &  &  &  &  & A_{r-1}^T &  &  \\
     \end{array}
   \right)$$
\bigskip
\noindent obtained by deleting the columns $\{\phi(i)n+i,i=0,1,\ldots,n-1\}$. Hence we get (up to the sign) the determinant of the matrix  obtained by $\left(
                                                                                                        \begin{array}{c}
                                                                                                          A_0^T \\
                                                                                                          A_1^T \\
                                                                                                          \vdots \\
                                                                                                          A_{r-2}^T \\
                                                                                                        \end{array}
                                                                                                      \right)
$ by taking the columns $\{i\in \Z_n\colon \phi(i)\neq r-1\}$ and the rows $\phi^{-1}(j)$ in $A_{j}^T$ $\forall \, j=0,1,\ldots,r-2$. We observe that if $\phi(i)=r-1$ for all $i \in \{0,1,\ldots,n-1\}$, then we get 1 as coefficient of $N(x_{r-1})=x_{r-1}^{1+q+\cdots + q^{n-1}}$. So for every $\phi$ not constantly equal to $r-1$, let $I:=\{i \in\{0,1,\ldots,n-1\}|\phi(i) \neq r-1\}$ and let $\psi: I \to \{0,1,\ldots,r-2\}$ be such that $\psi(i)=\phi(i)$ $\forall \, i \in I$, then the coefficient of  $\prod_{i=0}^{n-1} x_{\phi(i)}^{q^i}$  (up to the sign) is $\det C_{\psi}$. We can get $X(W)=V(G)$ in the same way and by Theorem \ref{main1}, $G=\lambda F$ for some $\lambda \in \F_{q^n}^*$. Since for both $F$ and $G$ the coefficient of $x_{r-1}^{1+q+\cdots + q^{n-1}}$
 is 1, we must have $F=G$ and hence the statement.
\end{proof}

In the following, by \emph{principal minors} we will mean the minors of a square matrix obtained by considering the same subset of indexes for the rows and the columns.

For $r=2$, Theorem \ref{main2} reads as follows.

\begin{theorem}\label{main3}
Let $U$ and $W$ be two $\F_q$-vector spaces of dimension $n$ such that $L_U$ and $L_W$ do not contain $P_{\infty}=(0,1)$. Let $A$ and $B$  be the Dickson matrices associated to $U$ and $W$ respectively. Then $L_U=L_W$ if and only if $A$ and $B$ have equal corresponding principal minors.
\end{theorem}

\section{Dickson matrices with equal corresponding principal minors}

As it is shown in the previous section, vectors subspaces determining the same linear set of the projective line correspond to Dickson matrices with equal corresponding principal minors.

In the following,  we will index the rows and columns  of an $n \times n$ matrix by $\Z_n$.
Let $\alpha, \beta$ be non-empty subsets of $\Z_n$. Then by $A[\alpha|\beta]$ we denote the submatrix of a matrix $A$ determined by the rows in the set $\alpha$ and the columns in the set $\beta$.

The matrix $A$ is \emph{reducible} if, up to a permutation of the rows and the columns, it is a block upper triangular matrix, that is equivalent to say that if there is a non-trivial partition $\{\alpha, \beta\}$ of $\Z_n$  such that $A[\alpha|\beta]$ is an all-zero matrix.

A matrix $B$ is \emph{diagonally similar} to $A$ if and only if there exists an invertible diagonal matrix $D$ such that $B=D^{-1}AD$. Diagonally similar matrices clearly have equal corresponding principal minors. It has been investigated when that condition is necessary.

The most general results about that are the following.

\begin{theorem}\cite[Th. 3]{HartfLoewy}\label{sim1}
Let $A$ be an irreducible $n \times n$ matrix with entries in a field $\F$,  $n \leq 3$, and let $B$ be a matrix of the same order such that $A$ and $B$ have equal corresponding principal minors. Then $A$ is diagonally similar to $B$ or to $B^T$.
\end{theorem}

\begin{theorem}\cite{loewy}\label{sim2}
  Let $A$, $B$ be $n\times n$ matrices with entries in a field $\F$ and $n\geq 4$. Let $A$ be irreducible and such that for every partition of $\Z_n$ into subsets $\alpha,\beta$ with $|\alpha|\geq 2$, $|\beta| \geq 2$ then either $\rank A[\alpha|\beta]\geq 2$ or $\rank A[\beta|\alpha]\geq 2$. If $A$ and $B$ have equal corresponding principal minors, of all orders, then $A$  is diagonally similar to $B$ or $B^T$.
\end{theorem}

So we introduce the following definition.

\begin{definition}
  A square matrix $A$ of order $n$ such that there exists a partition of $\mathbb{Z}_n$ into subsets $\alpha,\beta$ with $|\alpha|\geq 2$, $|\beta| \geq 2$ with $\rank A[\alpha|\beta]=\rank A[\beta|\alpha]=1$ is said to be  Loewy-reducible.
\end{definition}

We will denote by $diag(d_0,d_1,\cdots,d_{n-1})$ a diagonal matrix $D$ such that $D[i|i]=d_i$ $\forall \, i \in \Z_n$.

We will address the problem for Dickson matrices.

Let $A$ be the Dickson matrix associated to $f(x)=\displaystyle \sum_{i=0}^{n-1}a_ix^{q^i}$, that is $A$ is an $n \times n$ matrix such that $A[i|j]=a_{j-i}^{q^i}$. Hence

$$
A=\left(
    \begin{array}{cccc}
      a_0 & a_1 & \cdots & a_{n-1} \\
      a_{n-1}^q & a_0^q &  \cdots & a_{n-2}^q \\
       \cdots &  \cdots & \cdots &  \cdots \\
      a_1^{q^{n-1}} & a_2^{q^{n-1}} &  \cdots & a_0^{q^{n-1}} \\
    \end{array}
  \right)
$$

and

$$
A^T=\left(
    \begin{array}{cccc}
      a_0 & a_{n-1}^q & \cdots & a_1^{q^{n-1}} \\
      a_1 & a_0^q &  \cdots & a_2^{q^{n-1}} \\
       \cdots &  \cdots & \cdots &  \cdots \\
      a_{n-1} &  a_{n-2}^q&  \cdots & a_0^{q^{n-1}} \\
    \end{array}
  \right).
$$
 Also, we have that $A[\alpha|\beta]^q=A[\alpha+1|\beta+1]$ for all subsets $\alpha, \beta$ of $\Z_n$, hence $\rank A[\alpha|\beta]=\rank A[\alpha+h|\beta+h]$ $\forall \, h \in \Z_n$. The matrix

 \noindent $B=$ $diag(\lambda,\lambda^q,\ldots,\lambda^{q^{n-1}})^{-1}\, A \, diag(\lambda,\lambda^q,\ldots,\lambda^{q^{n-1}})$ is a Dickson matrix too and $B[0|i]=a_i\lambda^{q^i-1}$.

We will need the following result about the rank of Dickson matrices.

\begin{theorem}\label{rank}
Let $A$ be a Dickson matrix and let $A_{k-1}$ be $A[0,1,\ldots,k-1|0,1,\ldots,k-1]$. Then $\rank A= \max \{k : \det A_{k-1} \neq 0\}$.
\end{theorem}
\begin{proof}
Let $\tau$ be the $\F_q$-linear map $(a_0,a_1,\ldots,a_{n-1})\mapsto (a_{n-1},a_1,\ldots,a_{n-2})^q$ of $\F_{q^n}^n$ and let $\mathbf{a}$ be $(a_0,a_1,\ldots,a_{n-1})$. Then the rows of $A$ form a full orbit under the action of $\tau$ of a vector $\mathbf{a} \in \F_{q^n}^n$ and $\rank A = \dim\langle \mathbf{a}^{\tau^i},i=0,1,\ldots, n-1 \rangle$. Let $\rank A = k >0$. Then, by \cite{CS}, $\det A[0,1,\ldots,k-1|n-k,n-k+1,\ldots,n-1] \neq 0$. Hence, the first $k$ rows of $A$ are linearly independent, that is

$$\dim\langle \mathbf{a}^{\tau^i},i=0,1,\ldots, k-1 \rangle=k.$$

\noindent Let $\F_{q^n} = W_1 \oplus W_2$, where $W_1=\{(x_0,x_1,\ldots,x_n) \in \F_{q^n} : x_i=0 \, \forall \, i \leq k-1\}$ and $W_2= \{(x_0,x_1,\ldots,x_n) \in \F_{q^n} : x_i=0 \, \forall \, i \geq k\}$. Suppose that $\det A_{k-1}=0$, then the projection of $\langle \mathbf{a}^{\tau^i},i=0,1,\ldots, k-1 \rangle$ from $W_2$ onto $W_1$ has dimension lower than $k$, that is $\langle \mathbf{a}^{\tau^i},i=0,1,\ldots, k-1 \rangle \cap W_2 \neq 0$. Let $\mathbf{b} \in \langle \mathbf{a}^{\tau^i},i=0,1,\ldots, k-1 \rangle \cap W_2$, $\mathbf{b} \neq \mathbf{0}$. Since $\langle \mathbf{a}^{\tau^i},i=0,1,\ldots, n-1 \rangle$ is set-wise fixed by $\tau$, we get

$$\langle \mathbf{b}^{\tau^i},i=0,1,\ldots, n-1 \rangle \subseteq \langle \mathbf{a}^{\tau^i},i=0,1,\ldots, n-1 \rangle.$$

\noindent The non-zero vector $\mathbf{b}$ is in $W_2$, hence it  has at least $k$ consecutive zero components. Therefore, it is easy to see that the rank of the Dickson matrix having the orbit of $\mathbf{b}$ as rows has rank at least $k+1$. Hence,

$$ k+1 \leq \dim \langle \mathbf{b}^{\tau^i},i=0,1,\ldots, n-1 \rangle \leq \dim \langle \mathbf{a}^{\tau^i},i=0,1,\ldots, n-1 \rangle,$$

\noindent a contradiction.
\end{proof}

\begin{corollary}\label{io}
  Dickson matrices with equal corresponding principal minors have the same rank.
\end{corollary}
\begin{proof}
  It is an easy application of Theorem \ref{rank}.
\end{proof}

Let $d$ be a divisor of $n$, by $d\mathbb{Z}_n$ we will denote the subgroup of index $d$ of $(\mathbb{Z}_n,+)$, that is $d\mathbb{Z}_n=\{0,d,2d,\ldots,n-d\}$.

A function $f: x \in \F_{q^n} \mapsto \displaystyle \sum_{i=0}^{n-1}a_ix^{q^i} \in \F_{q^n}$ is clearly $\F_q$-linear, but we will say that $\F_{q^d}$ is the \emph{maximum field of linearity} for $f(x)$ if $d$ is the largest divisor of $n$ such that $f(x)$ is $\F_{q^d}$-linear.

\begin{proposition}\label{rank0}
  If the Dickson matrix $A$ associated to $f(x)$ is reducible, then  $f(x)$ is $\F_{d}$-linear, for some $d|n$, $d >1$.
\end{proposition}
\begin{proof}
If there exists a non-trivial partition of $\mathbb{Z}_n$ into subsets $\alpha,\beta$  such that $\rank A[\alpha|\beta]=0$, then $a_{j-i}=0$ for all $i \in \alpha$ and $j \in \beta$. If $a_h=0$ for all $h \neq 0$, then $f(x)$ is $\F_{q^n}$-linear. Suppose that there exists an $h \in \Z_n \setminus \{0\}$ such that $a_h \neq 0$, hence $h$ cannot be written as $h=j-i$ for some $i \in \alpha$ and $j \in \beta$. Therefore, $h+\alpha=\alpha$ and hence there exists a non-trivial stabilizer of $\alpha$ in $(\Z_n,+)$, say $ d\mathbb{Z}_n$ for some $d\mid n$ and $h \in d\mathbb{Z}_n$. If $d=1$, then $\alpha=\Z_n$, a contradiction to our hypothesis. Then $d>1$ and $f(x)$ turns out to be $\F_{q^d}$-linear.
\end{proof}

\begin{proposition}\label{diagonal}
If $A$ and $B$ are diagonally similar Dickson matrices, then

$$B=diag(\lambda,\lambda^q,\ldots,\lambda^{q^{n-1}})^{-1}A \, diag(\lambda,\lambda^q,\ldots,\lambda^{q^{n-1}})$$

\noindent for some $\lambda \in \F_{q^n}^*$, and the associated linearized polynomials $f(x)$ and $g(x)$ are such that $g(x)=\lambda^{-1}f(\lambda x)$. If $\F_{q'}$, $\F_q \subseteq \F_{q'} \subseteq \F_{q^n}$, is the maximum field of linearity for $f(x)$ and $g(x)$, then $\lambda$ is unique modulo $\F_{q'}^*$.
\end{proposition}
\begin{proof}
The statement is trivially true if $\F_{q'} = \F_{q^n}$. Suppose that $\F_{q'} \subsetneq \F_{q^n}$ and hence $A$ is not diagonal.
Let $D=diag(\lambda_0,\lambda_1,\ldots,\lambda_{n-1})$ with $\lambda_i \neq 0 \, \forall \, i=0,1,\ldots,n-1$, and let $B=D^{-1}AD$, hence $B[i|j]=a_{j-i}^{q^i}\lambda_i^{-1}\lambda_j$. By $B[i+h|j+h]=B[i|j]^{q^h}$, we get $a_{j-i}^{q^{i+h}}\lambda_{i+h}^{-1}\lambda_{j+h}=a_{j-i}^{q^{i+h}}(\lambda_i^{-1}\lambda_j)^{q^h}$, so for every $i,j$ such that $a_{j-i} \neq 0$,
$\det \left(                                                                                                                                                               \begin{array}{cc}
\lambda_{i+h} & \lambda_{j+h} \\
\lambda_i^{q^h} & \lambda_j^{q^h} \\
\end{array}
\right)=0
$.

Let $a_h\neq 0$ for some $h \neq 0$ and let $d\mathbb{Z}_n$ be the subgroup of $(\Z,+)$ generated by $h$.  Let $M=\left(
         \begin{array}{cccc}
           \lambda_0 & \lambda_{d} & \cdots & \lambda_{n-d} \\
           \lambda_{n-d}^{q^d}& \lambda_{0}^{q^d} & \cdots & \lambda_{n-2d}^{q^d} \\
         \end{array}
       \right)
$, then $\rank M=1$.  Therefore, $(\lambda_0,\lambda_d,\ldots,\lambda_{n-d})=\mu (\lambda,\lambda^{q^d},\ldots,\lambda^{q^{n-d}})$ for some $\mu, \lambda \in \F_{q^n}^*$. Hence,

$$B[d\mathbb{Z}_n|d\mathbb{Z}_n]=diag(\lambda,\lambda^{q^d},\ldots,\lambda^{q^{n-d}})^{-1} A[d\mathbb{Z}_n|d\mathbb{Z}_n]\, diag(\lambda,\lambda^{q^d},\ldots,\lambda^{q^{n-d}}).$$

\noindent By $B[d\mathbb{Z}_n+i|[d\mathbb{Z}_n+i]=B[d\mathbb{Z}_n|[d\mathbb{Z}_n]^{q^i}$ $\forall \, i=1,2,\ldots,d-1$, we get

$$B=diag(\lambda,\lambda^q,\ldots,\lambda^{q^{n-1}})^{-1}A \, diag(\lambda,\lambda^q,\ldots,\lambda^{q^{n-1}}).$$

\noindent Suppose that there exist $\lambda_1,\lambda_2$ such that

$$B=diag(\lambda_i,\lambda_i^q,\ldots,\lambda_i^{q^{n-1}})^{-1}A \, diag(\lambda_i,\lambda_i^q,\ldots,\lambda_i^{q^{n-1}})$$

\noindent for $i=1,2$. Then $\lambda_1^{q^j-1}=\lambda_2^{q^j-1}$ $\forall \, j$ such that $a_{j} \neq 0$, hence $\lambda_2=\nu \lambda_1$, where $\nu \in \F_{q'}^*$.
\end{proof}

Let $A'$ be $A+diag(d_0,d_1,\cdots,d_{n-1})$. If $\alpha$ and $\beta$ are disjoint subsets of $\Z_n$, then $A[\alpha|\beta]=A'[\alpha,\beta]$. Hence, if $\rank A' = 1$, then $A$ is  Loewy-reducible. A rank 1 Dickson matrices is of the form:

$$
A=\left(
    \begin{array}{cccc}
      \lambda a & \lambda a^q & \cdots & \lambda a^{q^{n-1}} \\
       \lambda^q a & \lambda^q a^q & \cdots & \lambda^q a^{q^{n-1}} \\
      \cdots  & \cdots  & \cdots  & \cdots \\
     \lambda^{q^{n-1}} a & \lambda^{q^{n-1}} a^q & \cdots & \lambda^{q^{n-1}} a^{q^{n-1}} \\
    \end{array}
  \right)
=\left(
   \begin{array}{c}
     \lambda \\
     \lambda^q \\
     \dots \\
     \lambda^{q^{n-1}} \\
   \end{array}
 \right)\left(
          \begin{array}{c}
            a,  a^q, \cdots, a^{q^{n-1}} \\
          \end{array}
        \right)
$$

\noindent with $a\lambda \neq 0$.
We will show that for rank 1 Dickson matrices, in Theorem \ref{sim2} the hypothesis of being irreducible can be dropped.

\begin{proposition}\label{club}
Let $A$   be a Dickson matrix such that $\rank (A+diag(d,d^q,\cdots,d^{q^{n-1}}))=1$ and let $B$ a Dickson matrix such that  $A$ and $B$ have equal corresponding principal minors, of all orders, then $A$  is diagonally similar to $B$.
\end{proposition}
\begin{proof}
The matrices $A'=A+diag(d,d^q,\cdots,d^{q^{n-1}})$ and $B'=B+diag(d,d^q,\cdots,d^{q^{n-1}})$ have have equal corresponding principal minors too.  By Corollary \ref{io}, $\rank B'=\rank A'=1$, hence

$$A'=\left(
   \begin{array}{c}
     \lambda \\
     \lambda^q \\
     \dots \\
     \lambda^{q^{n-1}} \\
   \end{array}
 \right)\left(
          \begin{array}{c}
            a, a^q, \cdots, a^{q^{n-1}} \\
          \end{array}
        \right)$$

  and

  $$B'=\left(
   \begin{array}{c}
     \mu \\
     \mu^q \\
     \dots \\
     \mu^{q^{n-1}} \\
   \end{array}
 \right)\left(
          \begin{array}{cccc}
            b, b^q, \cdots, b^{q^{n-1}} \\
          \end{array}
        \right)$$

\noindent with $\lambda a\neq 0 \neq \mu b$. Then

 $$
 diag(\lambda,\lambda^q,\cdots,\lambda^{q^{n-1}})^{-1}A' diag(\lambda,\lambda^q,\cdots,\lambda^{q^{n-1}})= \left(
    \begin{array}{cccc}
      \lambda a & \lambda^q a^q & \cdots & \lambda^{q^{n-1}} a^{q^{n-1}} \\
       \lambda a & \lambda^q a^q & \cdots & \lambda^{q^{n-1}} a^{q^{n-1}} \\
      \cdots  & \cdots  & \cdots  & \cdots \\
     \lambda a & \lambda^{q} a^q & \cdots & \lambda^{q^{n-1}} a^{q^{n-1}} \\
    \end{array}
  \right)
 $$

 \noindent and

  $$
 diag(\mu,\mu^q,\cdots,\mu^{q^{n-1}})^{-1}B' diag(\mu,\mu^q,\cdots,\mu^{q^{n-1}})= \left(
    \begin{array}{cccc}
      \mu b & \mu^q b^q & \cdots & \mu^{q^{n-1}} b^{q^{n-1}} \\
       \mu b & \mu^q b^q & \cdots & \mu^{q^{n-1}} b^{q^{n-1}} \\
      \cdots  & \cdots  & \cdots  & \cdots \\
     \mu b & \mu^{q} b^q & \cdots & \mu^{q^{n-1}} b^{q^{n-1}} \\
    \end{array}
  \right)
 $$

\noindent have equal principal minors too, hence the same diagonal, so $\lambda a =\mu b$.   Therefore,

\[ diag(\lambda,\lambda^q,\cdots,\lambda^{q^{n-1}})^{-1}A' \, diag(\lambda,\lambda^q,\cdots,\lambda^{q^{n-1}})=\]
\[ diag(\mu,\mu^q,\cdots,\mu^{q^{n-1}})^{-1}B' \, diag(\mu,\mu^q,\cdots,\mu^{q^{n-1}}),\]

 \noindent hence

\[ diag(\lambda,\lambda^q,\cdots,\lambda^{q^{n-1}})^{-1}A \, diag(\lambda,\lambda^q,\cdots,\lambda^{q^{n-1}})=\]
 \[diag(\mu,\mu^q,\cdots,\mu^{q^{n-1}})^{-1}B \, diag(\mu,\mu^q,\cdots,\mu^{q^{n-1}}),\]

 \noindent and $B$ is diagonally similar to $A$.
\end{proof}

\medskip

The hypothesis of being Loewy-irreducible can be dropped also for $n=4$.

\begin{theorem}\label{quattro}
  Let $A$ be a Dickson irreducible matrix of order 4, and let $B$ a Dickson matrix such that $A$ and $B$ have equal corresponding principal minors. Then $B$ is diagonally similar to $A$ or $A^T$.
\end{theorem}
\begin{proof}
Let $a_i:=A[0|i]$ and $b_i:=B[0|i]$, $i=0,1,2,3$.  By $A[0|0]=B[0|0]$, we get $b_0=a_0$. By $\det A[0,2|0,2]= \det B[0,2|0,2]$, we get $b_2^{q^2+1}=a_2^{q^2+1}$, hence $b_2=u^{q^2-1}a_2$ for some $u \in \F_{q^4}^*$. Then $B'=\, diag(u,u^q,u^{q^2},u^{q^3}) \, B \, diag(u,u^q,u^{q^2},u^{q^3})^{-1}$ has $B'[0|2]=a_2$, so we replace $B$ by $B'$. By $\det A[0,1|0,1]=\det B[0,1|0,1]$, we get
\begin{equation}\label{uno}
 a_1a_3^q=b_1b_3^q
\end{equation}
 and by $\det A[0,1,2|0,1,2]= \det B[0,1,2|0,1,2]$, we get
\begin{equation}\label{due}
a_1^{q+1}a_2^{q^2}+a_2a_3^{q+q^2}=b_1^{q+1}a_2^{q^2}+a_2b_3^{q+q^2}.
\end{equation}
Let $a_2 \neq 0$. Suppose that $b_1 \neq 0$, then by \rif{uno} $b_3^q=\frac{a_1a_3^q}{b_1}$. Plugging that in \rif{due}, we get $b_1^{q+1} \in \{a_1^{q+1},a_2^{1-q^2}a_3^{q+q^2}\}$. If $b_1^{q+1}=a_1^{q+1}$, then $b_1=v^{q-1}a_1$ and $b_3=v^{q-1}a_3$ for some $v \in \F_{q^2}^*$, hence $B=\, diag(v,v^q,v,v^{q}) \, A \, diag(v,v^q,v,v^{q})^{-1}$. If $b_1^{q+1}=a_2^{1-q^2}a_3^{q+q^2}$, then $b_1=v^{q-1}a_2^{1-q}a_3^q$ and $b_3=v^{q-1}a_1^{q^3}a_2^{1-q^3}$ for some $v \in \F_{q^2}^*$, hence  $B=\, diag(\frac{v}{a_2},(\frac{v}{a_2})^q,(\frac{v}{a_2})^{q^2},(\frac{v}{a_2})^{q^3}) \, A^T \, diag(\frac{v}{a_2},(\frac{v}{a_2})^q,(\frac{v}{a_2})^{q^2},(\frac{v}{a_2})^{q^3})^{-1}$. Suppose that $b_1=0$, then $a_1a_3=0$ by \rif{uno}. If $a_1=0$, then, by \rif{due}, $b_3=v^{q-1}a_3$ for some $v \in \F_{q^2}^*$, hence $B=\, diag(\frac{v}{a_2},(\frac{v}{a_2})^q,(\frac{v}{a_2})^{q^2},(\frac{v}{a_2})^{q^3}) \, A^T \, diag(\frac{v}{a_2},(\frac{v}{a_2})^q,(\frac{v}{a_2})^{q^2},(\frac{v}{a_2})^{q^3})^{-1}$. If $a_3=0$, then, by \rif{due}, $b_3=v^{1-q}a_1^{q^3}a_2^{1-q^3}$ for some $v \in \F_{q^2}^*$, hence  $B=$

\noindent $diag(\frac{v}{a_2},(\frac{v}{a_2})^q,(\frac{v}{a_2})^{q^2},(\frac{v}{a_2})^{q^3})\, A^T \, diag(\frac{v}{a_2},(\frac{v}{a_2})^q,(\frac{v}{a_2})^{q^2},(\frac{v}{a_2})^{q^3})^{-1}$. Let $a_2=0$, then $(a_1,a_3)\neq (0,0)$ as $\F_q$ is the maximum field of linearity for the linear function associated to $A$. By $\det A=\det B$, $a_0=b_0$ and \rif{uno}, we get
\begin{equation}\label{tre}
  N_{q^4|q}(a_1)+ N_{q^4|q}(a_3)= N_{q^4|q}(b_1)+ N_{q^4|q}(b_3).
\end{equation}
By \rif{uno} and \rif{tre}, then also $(b_1,b_3)\neq (0,0)$. Let $b_1 \neq 0$, then  by \rif{uno} $b_3^q=\frac{a_1a_3^q}{b_1}$.  Plugging that in \rif{due}, we get $N_{q^4|q}(b_1)=N_{q^4|q}(a_1)$ or $N_{q^4|q}(b_1)=N_{q^4|q}(a_3)$. In the first case we get $b_1=a_1v^{q-1}$, $b_3=a_3v^{q^3-1}$ for some $v \in \F_{q^4}^*$, hence $B=$ $\, diag(v,v^q,v^{q^2},v^{q^3})^{-1} \, A \, diag(v,v^q,v^{q^2},v^{q^3})$. If $N_{q^4|q}(b_1)=N_{q^4|q}(a_3)$, then $b_1=a_3v^{q-1}$, for some $v \in \F_{q^4}^*$. Hence $a_3\neq 0$ and we can set $t:=\frac{v}{a_3}$. So  $b_1=a_3^qt^{q-1}$ and $b_3=a_1^{q^3}t^{q^3-1}$ for $t \in \F_{q^4}^*$, hence
$B=\, diag(t,t^q,t^{q^2},t^{q^3})^{-1} \, A^T \, diag(t,t^q,t^{q^2},t^{q^3})$.
If $b_3\neq 0$, just replace $A$ and $B$ by $A^T$ and $B^T$ respectively.
\end{proof}

We will need the following result to show the feature of Loewy-reducible Dickson matrices.

\begin{lemma}{\cite[Lemma 1]{loewy}}\label{reduction}
 Let $A$ be an irreducible matrix of order $n$ over any field. Suppose that $\{\alpha, \beta\}$ and $\{\gamma,\delta\}$ are partitions of $\mathbb{Z}_n$ such that $\alpha \cap \gamma \neq \emptyset$ and $\beta \cap \delta \neq \emptyset$. If $\rank A[\alpha, \beta]=\rank A[\gamma|\delta]=1$, then $\rank A[\alpha \cap \gamma|\beta \cup \delta]=\rank A[\alpha \cup \gamma|\beta \cap \delta]=1$.
\end{lemma}

\begin{theorem}\label{classification}
Let $A$ be a Loewy-reducible Dickson matrix with respect to the partition $\{\alpha,\beta\}$, $|\alpha| \leq |\beta|$, of $\Z_n$. Then there exists a partition $\{\gamma, \delta\}$ of $\Z_n$ such that either  $\gamma = d\Z_n$, for some $1<d\mid n$, or $|\gamma|=2$,  with respect to which $A$ is Loewy-reducible.
\end{theorem}
\begin{proof}
 By $\rank A[\alpha|\beta] = \rank A[\beta|\alpha] = 1$, we get  $\rank A[\alpha + h|\beta + h] = $

  \noindent$ \rank A[\beta + h|\alpha + h] = 1$ $\forall \, h \in \Z_n$ and $\{\alpha+h,\beta+h\}$ is also a partition of $\Z_n$. Let $h \in \Z_n$ such that the map $m \mapsto m+h$ does not fix $\alpha$ (hence $\beta$), and such that $\alpha \cap \alpha + h \neq \emptyset \neq \beta \cap \beta +h$. By  $\rank A[\alpha + h|\beta + h] =$ $\rank A[\alpha|\beta]=1$ and Lemma \ref{reduction}, we get $\rank A[\alpha \cap \alpha +h|\beta \cup \beta +h]=1$. By $\rank A[\beta + h|\alpha + h] = \rank A[\beta|\alpha] =1$ and Lemma \ref{reduction}, we get
$\rank A[\beta \cup \beta +h|\alpha \cap \alpha+h]=1$. If $|\alpha \cap \alpha+h| \geq 2$, then $A$ is  Loewy-reducible with respect to $\{\alpha \cap \alpha +h,\beta \cup \beta+h\}$. If $|\alpha \cap \alpha+h|=1$, then $|\alpha \cap \beta+h|=|\alpha|-1$. Also, $\beta\cap \alpha+h \neq \emptyset$ as $\alpha+h \neq \alpha$. By  $\rank A[\beta + h|\alpha + h] =\rank A[\alpha|\beta]=1$ and Lemma \ref{reduction}, we get $\rank A[\alpha \cap \beta +h|\beta \cup \alpha +h]=1$. By $\rank A[\alpha + h|\beta + h] = \rank A[\beta|\alpha] =1$ and Lemma \ref{reduction}, we get $\rank A[\beta \cup \alpha +h|\alpha \cap \beta+h]=1$. If $|\alpha|-1=1$, then $|\alpha|=2$ as in the statement. If $|\alpha|-1 \geq 2$, then $A$ is Loewy-reducible with respect to the partition $\{\alpha\cap \beta +h,\beta \cup \alpha +h\}$. We can apply the reduction step until we get an $\alpha$ such that $\alpha=2$, unless  $\forall \, h \in \Z_n$ not fixing $\{\alpha,\beta\}$, either  $\alpha+h\cap\alpha=\emptyset$ or $\beta\cap \beta+h=\emptyset$. If $\beta\cap \beta+h=\emptyset$, then $\beta+h=\alpha$, hence $|\alpha|=|\beta|=\frac{n}{2}$ and $\alpha=2\Z_n$. If $\alpha+h=\alpha$ or $\alpha +h \cap \alpha=\emptyset$ $\forall \, h \in \Z_n$, then there exists a non-trivial stabilizer of $\alpha$, say $d\Z_n$, such that $\forall \, h \notin d\Z_n$, $\alpha+h\subset \beta$. Let $\overline{\alpha}$ and $\overline{\beta}$ be the image of $\alpha$ and $\beta$ in $\Z_d$, then, $\overline{\alpha}+i=\overline{\alpha}$ only if $i=0$ and  $\overline{\alpha}+i\subset \beta$ $\forall \, i \in \Z_d\setminus \{0\}$. Hence $|\overline{\alpha}|=1$, so $\alpha$ consists of one orbit under the action of $(\Z_d,+)$. Up the action of $(\Z_n,+)$, we can assume that $\alpha=d\Z_n$. In this case, we can not apply Lemma \ref{reduction} to $\{\alpha,\beta\}$ and $\{\alpha+h,\beta+h\}$ nor to  $\{\alpha,\beta\}$  and  $\{\beta+h,\alpha+h\}$.
\end{proof}

\begin{proposition}\label{divisor}
Let $A$ be a  Loewy-reducible Dickson matrix with respect to the partition $\{d\Z_n,\Z_n\setminus d\Z_n\}$, $1<d \mid n$, then the associated linear function is

$$f(x)+\lambda \displaystyle \sum_{i=1}^{d-1}b_iTr_{\F_{q^n}|\F_{q^d}}(ax)^{q^i}$$

\noindent for some $\lambda,a \in \F_{q^n}^*$, $b_i \in \F_{q^d}$ not all zero and $f(x)$ $\F_{q^d}$-linear.
\end{proposition}
\begin{proof}
Let $a_i:=A[0|i]$. By $\rank A[d\Z_n|i+d\Z_n]\leq 1$, we get $(a_i,a_{i+d},\ldots,a_{i+n-d})=$ $\lambda_i(a_i,a_i^{q^d},\ldots,a_i^{q^{n-d}})$ $\forall \, i=1,2,\ldots,d-1$. If $a_i=0$ for all $i=1,2,\ldots,d-1$, then  $\rank A[dZ_n|\Z_n\setminus d\Z_n]=0$, a contradiction. Hence there exists $i \in \{1,2,\ldots,d-1\}$ such that $a_i \neq 0$. Also, $\rank A[0,d|1,2,\ldots,d-1]= 1$, and hence $\lambda_j=\lambda_i$ for all $j$ such that $a_j \neq 0$. So we get $(a_i,a_{i+d},\ldots,a_{i+n-d})=$ $\lambda(a_i,a_i^{q^d},\ldots,a_i^{q^{n-d}})$ $\forall \, i=1,2,\ldots,d-1$ for some $\lambda \in \F_{q^n}^*$. If $d=2$, then the associated linear function is $f(x)+\lambda Tr_{q^n|q^2}(a_1^{q^{n-1}}x)^q$. Let $d>2$.  By $\rank A[1,2,\ldots,d-1|d\Z_n]=1$, we get $(a_1,a_2,\ldots,a_{d-1})=(a^qb_1,a^{q^2}b_2,\ldots,a^{q^{d-1}}b_{d-1})$ for some $a \in \F_{q^n}$, $b_i \in \F_{d}$, $i=1,2,\ldots,d-1$. Hence the linear function associated to $A$ is  $f(x)+\lambda \displaystyle \sum_{i=1}^{d-1}b_iTr_{\F_{q^n}|\F_{q^d}}(ax)^{q^i}$, for some $f(x)$ $\F_{q^d}$-linear. If $\lambda a=0$, or  $b_i=0 \forall \, i=1,2,\ldots,d-1$, then $\rank A[dZ_n|\Z_n\setminus d\Z_n]=0$, a contradiction.
\end{proof}

\section{Proof of the main results}

Let $r=2$. Therefore $\dim_{\F_q} U=n$. The linear set $L_U$ is regarded as the set of projective points

	\[\{\la {\bf u} \ra_{\F_{q^n}} \colon {\bf u}\in U\setminus \{{\bf 0} \}\}\subset \PG(1,q^n)\]

\noindent as well as the spread elements of $\PG(2n-1,q)$

\[\Pi_P=\{\lambda (a,b), \lambda \in \F_{q^n} \colon \Pi_P \cap \PG(U,q) \neq \emptyset.\]

\noindent Let $\perp$ be the polarity induced by

\[\mathbf{b}:((x_1,y_1),(x_2,y_2)) \in \F_{q^n}^2 \mapsto Tr_{q^n|q}(x_1y_2-x_2y_1)\in \F_q.\]
\medskip

\noindent Then, the group PSL$(2,q^n)$ is a subgroup of the symplectic group of $\perp$, hence, by the transitivity of   PSL$(2,q^n)$, we can assume \textit{without loss of generality} that the projective point $P_{\infty}(0,1)$ is not in  $L_U$ in the remaining of the Section. Hence, $U=\{(x,f(x)), x \in \F_{q^n}\}$ and $f(x)=\displaystyle \sum_{i=0}^{n-1}a_ix^{q^i}$. We now also assume that $\F_{q}$ is the maximum field of linearity for $L_U$, hence, the greatest common divisor of the integers $\{ i| a_i \neq 0 \}$ must be 1. By \cite{CsMP23}, that implies that there exists at least a point of weight 1 in $L_U$. Let $A$ be the Dickson matrix associated to $f(x)$. Then, by Proposition \ref{rank0}, $A$ is irreducible.

For the sake of completeness, we include here a proof of Theorem \ref{main3} not involving the variety $\mathcal{V}_{2n}$ but only Dickson matrices.

\begin{proof}
Let $U=\{(x,f(x)), x \in \F_{q^n}\}$, $W=\{(x,g(x)), x \in \F_{q^n}\}$ and let $A$ and $B$ be the Dickson matrices associated to $f(x)$ and $g(x)$ respectively. We have $L_U=L_W$. A point $P=(1,a)$ belongs to $L_U$ if and only if $f(x)=ax$ for some non-zero $x \in \F_{q^n}$, hence we must have $\det (A-diag(a,a^q,\ldots,a^{q^{n-1}}))=0$. We easily get

\begin{equation}\label{carpol}
p_A(\lambda)=\det (A-diag(\lambda,\lambda^q,\ldots,\lambda^{q^{n-1}}))=\displaystyle \sum_{I\subseteq \Z_n}\displaystyle\prod_{i \in I}(-1)^{|I|}\lambda^{q^i}\det A[\Z_n\setminus I|\Z_n \setminus I]
\end{equation}

\noindent and hence the leading term is $(-1)^n\lambda^{1+q+\cdots+q^{n-1}}$. The point $P=(1,a)$ has  weight $k$ if and only if $\rank (A-diag(a,a^q,\ldots,a^{q^{n-1}}))=n-k$  and hence, by Theorem \ref{rank}, if and only if
$\det(A[0,1,\ldots,j|0,1,\ldots,j]-diag(a,a^q,\ldots,a^{q^{j}})) =0$ $\forall \, j \geq n-k$ and
$\det(A[0,1,\ldots,n-k-1|0,1,\ldots,n-k-1]-diag(a,a^q,\ldots,a^{q^{n-k-1}})) \neq 0$. Let $A':= A-diag(a,a^q,\ldots,a^{q^{n-1}})$, then $\lambda=0$ is a root of $p_{A'}(\lambda)$ with multiplicity $1+q+q^2+\cdots + q^{k-1}$ and so is $\lambda = a$ for $p_A(\lambda)$. Since $|L_{U}|=\displaystyle \sum_{i=1}^{n}t_i$ and $\displaystyle \sum_{i=1}^{n}t_i\left(\frac{q^i-1}{q-1}\right)=1+q+\cdots + q^{n-1}$, where $t_i$ is the number of points of $L_{U}$ of weight $i$,  all the roots of $p_A(\lambda)$ are in $\F_{q^n}$. Since $p_A(\lambda)$ and $p_B(\lambda)$ have leading coefficient $(-1)^n$, they have the same roots if and only if $p_A(\lambda)= p_B(\lambda)$. By \rif{carpol}, we have the statement.
\end{proof}

If $B$ is diagonally similar to $A$, that is $B= diag(\lambda,\lambda^q,\ldots,\lambda^{q^{n-1}}) A \, diag(\lambda,\lambda^q,\ldots,\lambda^{q^{n-1}})^{-1}$, then, by Proposition \ref{diagonal},
the linear function associated to $B$ is $\lambda f(\lambda^{-1} x)$, hence $W=\{(x,\lambda f(\lambda^{-1} x), x \in \F_{q^n}\}= \lambda U$.

If $B$ is diagonally similar to $A^{T}$, and $W$ is the linear set associated to $B$, then $W=\lambda U^{\perp}$.

A \textit{club} $L_U$ is a rank $n$ linear set of $\PG(1,q^n)$ with one point of weight $n-1$, hence $U=\{(x,ax + \lambda Tr_{q^n|q}(bx)), x \in \F_{q^n}\}$ for some $\lambda, b \in \F_{q^n}^*$ and the projective point of weight $n-1$ is $P=(1,a)$.

In the following, we will denote by $W$ an $\F_q$-vector subspace of dimension $n$ of $V$.

By the results of the previous Section, we immediately get the following ones.

\begin{theorem}
For $n \leq 3$, $L_U=L_W$ if and only if $W=\lambda U$ or $W=\lambda U^{\perp}$ for some $\lambda \in \F_{q^n}^*$.
\end{theorem}
\begin{proof}
Apply Theorem \ref{sim1}.
\end{proof}

\begin{theorem}
For $n=4$, $L_U=L_W$ if and only if $W=\lambda U$ or $W=\lambda U^{\perp}$ for some $\lambda \in \F_{q^n}^*$.
\end{theorem}
\begin{proof}
Apply Theorem  \ref{quattro}.
\end{proof}

\begin{theorem}
If $L_U$ is a club, then $L_U=L_W$ if and only if $W=\lambda U$ for some $\lambda \in \F_{q^n}^*$.
\end{theorem}
\begin{proof}
Apply Proposition \ref{club}.
\end{proof}

We observe that the last three theorems are included in \cite{CsMP}.

We will need  the following result.

\begin{theorem}\label{hyperplane}
Let $U=\{(x,f(x)), x \in \F_{q^n}\}$, $W=\{(y,g(y)), y \in \F_{q^n}\}$,

\noindent $U'=\{(x,f(x)), x \in \F_{q^n}\colon Tr_{q^n|q}(ax)=0 \}$ and $W'=\{(y,g(y)), y \in \F_{q^n} \colon Tr_{q^n|q}(by)=0 \}$, with $a,b \in \F_{q^n}^*$. If $L_U=L_W$ and $L_{U'}=L_{W'}$, then $W=\frac{a}{b} U$.
\end{theorem}
\begin{proof}
We will adopt the cyclic representation of $\PG(2n-1,q)$ (see Section 2).

Let $A$ and $B$ the Dickson matrices associated to $U$ and $W$, respectively. By Theorem \ref{main3}, $A$ and $B$ have the same principal minors. By \rif{dimension}, we have that for every $P \in \Pi_0=\PG(1,q^n)$,

$$\la P^{\sigma^i}, i=0,1,\ldots, n-1\ra \cap U' \neq \mathbf{0} \Leftrightarrow \la P^{\sigma^i}, i=0,1,\ldots, n-1\ra \cap U'(\F_{q^n}) \neq \mathbf{0}.$$

\noindent Let $\overline{U'(\F_{q^n})}$ be the projection of $U'(\F_{q^n})$ from $\displaystyle \opl_{i \neq 0}V_i$ to $V_0$. Then, clearly, if  $\la P^{\sigma^i}, i=0,1,\ldots, n-1\ra \cap U'(\F_{q^n}) \neq \mathbf{0}$, then  $\la P^{\sigma^i}, i=0,1,\ldots, n-2 \ra \cap \overline{U'(\F_{q^n})} \neq \mathbf{0}$. On the other hand, by \cite[Lemma 4]{CsMP23}, if   $\la P^{\sigma^i}, i=0,1,\ldots, n-2 \ra \cap \overline{U'(\F_{q^n})} \neq \mathbf{0}$, then   $\la P^{\sigma^i}, i=0,1,\ldots, n-1\ra \cap U'(\F_{q^n}) \neq \mathbf{0}$. Hence,

$$\la P^{\sigma^i}, i=0,1,\ldots, n-1\ra \cap U' \neq \mathbf{0} \Leftrightarrow \la P^{\sigma^i}, i=0,1,\ldots, n-2 \ra \cap \overline{U'(\F_{q^n})} \neq \mathbf{0}.$$

\noindent We have that $U'(\F_{q^n})=\{(\mathbf{x}',\mathbf{x}'A^T), \mathbf{x}'=(x_0,x_1,\ldots,x_{n-2},-\frac{1}{a^{q^{n-1}}}\displaystyle\sum_{i=0}^{n-2}a^{q^i}x_i), x_i \in \F_{q^n}\}$. Let $M$ be the matrix

\bigskip

$$\left(
    \begin{array}{ccccc}
        &     a_0-a_{n-1} \frac{a}{a^{q^{n-1}}} & a_{n-1}^q-a_{n-2}^q \frac{a}{a^{q^{n-1}}} & \cdots & a_2^{q^{n-2}}  -a_1^{q^{n-2}}\frac{a}{a^{q^{n-1}}} \\
         &      a_1   -a_{n-1}\frac{a^q}{a^{q^{n-1}}} & a_0^q  -a_{n-2}^q \frac{a^q}{a^{q^{n-1}}} & \cdots & a_3^{q^{n-2}} -a_1^{q^{n-2}}\frac{a^q}{a^{q^{n-1}}} \\
        I_{n-1}  & \vdots & \vdots & \vdots & \vdots \\
         &   \vdots & \vdots & \vdots & \vdots \\
         &    a_{n-2}   -a_{n-1}\frac{a^{q^{n-2}}}{a^{q^{n-1}}} & a_{n-3}^q -a_{n-2}^q \frac{a^{q^{n-2}}}{a^{q^{n-1}}}   & \cdots & a_{0}^{q^{n-2}} -a_1^{q^{n-2}}\frac{a^{q^{n-2}} }{a^{q^{n-1}}}  \\
    \end{array}
  \right),
$$
\bigskip

\noindent that is, $M$ is an $(n-1)\times (2n-2)$ matrix such that its rows are a basis for $\overline{U'(\F_{q^n})}$. Then, $X(U')$ is defined by $\det N=0$, where $N$ is the matrix

 $$
\left(
     \begin{array}{cccccccc}
        &  &  &  &  &  &  &    \\
        &  &  & M &  &  &  &   \\
        &  &  &  &  &  &  &    \\
       x & 0&  \cdots & 0 & y & 0 &  \cdots & 0   \\
       0 & x^q &  \cdots & 0 & 0 & y^q &  \cdots & 0  \\
        \cdots &  \cdots &  \cdots &  \cdots &  \cdots &  \cdots &  \cdots &  \cdots \\
       0 & 0 &  \cdots & x^{q^{n-2}} & 0 & 0 &  \cdots & y^{q^{n-2}}  \\
     \end{array}
   \right).
 $$

\medskip

\noindent The same is true for $W'$, hence, as in Theorem \ref{main3}, $L_{U'}=L_{W'}$ if and only if

\medskip
 $A[\Z_n\setminus \{n-1\}|\Z_n\setminus \{n-1\}]^T-\frac{1}{a^{q^{n-1}}}\left(
                                       \begin{array}{c}
                                         a \\
                                         a^q \\
                                         \vdots \\
                                         a^{q^{n-2}} \\
                                       \end{array}
                                     \right)A[\Z_n\setminus \{n-1\}|n-1]^T$
                                     \medskip

\noindent and
\medskip
$B[\Z_n\setminus \{n-1\}|\Z_n\setminus \{n-1\}]^T-\frac{1}{b^{q^{n-1}}}\left(
                                       \begin{array}{c}
                                         b \\
                                         b^q \\
                                         \vdots \\
                                         b^{q^{n-2}} \\
                                       \end{array}
                                     \right)B[\Z_n\setminus \{n-1\}|n-1]^T$
\medskip

\noindent have the same principal minors. Hence the two matrices have the same diagonal, as well as $A$ and $B$, so

$$
a_0^{q^i}-a_{n-1-i}^{q^i}\frac{a^{q^i}}{a^{q^{n-1}}} = b_0^{q^i}-b_{n-1-i}^{q^i}\frac{b^{q^i}}{b^{q^{n-1}}}= a_0^{q^i}-b_{n-1-i}^{q^i}\frac{b^{q^i}}{b^{q^{n-1}}}, i=0,1,\ldots,n-2.
$$

Therefore, $b_i=a_i(\frac{b}{a})^{q^i-1}$, for $ i=0,1,\ldots,n-2$, hence

\noindent $B=diag (\frac{b}{a},(\frac{b}{a})^q,\ldots,(\frac{b}{a})^{q^{n-1}}) ^{-1} \, A \, diag (\frac{b}{a},(\frac{b}{a})^q,\ldots,(\frac{b}{a})^{q^{n-1}})$ and $W=\frac{a}{b}U$.

\end{proof}

From now on, we will assume that $n \geq 5$.

\begin{theorem}\label{pseudo}
Let $A$ is a  Loewy-reducible  Dickson matrix with respect to the partition $\{\alpha,\beta\}$ with $|\alpha|=2$, and $U=\{(x,f(x)), x \in \F_{q^n}\}$ the associated $\F_q$-vector space. Then one of the following possibilities must occur:
 \begin{enumerate}
   \item there exists $a \in \F_{q^n}$ such that $\rank (A - diag(a,a^q,\ldots,a^{q^n-1}))=1$ and hence $L_U$ is a club;
   \item $L_U$ is of pseudoregulus type;
   \item  $f(x)=ax+ \lambda(b_0Tr_{q^n|q^d}(cx)+b_1Tr_{q^n|q^d}(cx)^q+\cdots + b_{d-1}Tr_{q^n|q^d}(cx)^{q^{d-1}})$ for some $ 1 < d|n$ and $\lambda, c \in \F_{q^n}^*, b_i \in \F_{q^d} \, i=1,2,\ldots, d-1$ not all zero, and $(1,a)$ is a point of $L_U$ of weight at least $n-d$.
 \end{enumerate}
\end{theorem}
\begin{proof}
We will adopt the cyclic representation of $\PG(2n-1,q)$ (see Section 2). Let $R=(e_0,e_0A^{T})$, then $\la R,R^{\sigma},\ldots,R^{\sigma^{n-1}} \ra =U(\F_{q^n})$. Up the action of $(\Z_n,+)$, we can assume that $\alpha=\{0,i\}$, hence $\rank A[ j \neq 0,i|0,i]=1$. Hence, $\dim \la R,R^{\sigma^i} \ra \cap V_0\oplus V_i = \dim \la R,R^{\sigma^i} \ra - \rank A[ j \neq 0,i|0,i] = 1$ and there exists $v_0 \in V_0, v_i \in  V_i$ such that $v_0+v_i \in U(\F_{q^n})$. If one of the two vectors is $\mathbf{0}$, then $U(\F_{q^n})$ contains a non-zero vector $v_j$ of $V_j$ for some $j \in \{0,i\}$, hence $\la v_j,v_j^{\sigma},\ldots,v_j^{\sigma^{n-1}}\ra = U(\F_{q^n})$ and $U= \la v_j \ra_{\F_{q^n}}$, that is a $\F_{q^n}$-vector space, a contradiction to the hypothesis that $\F_q$ is the maximum field of linearity. Hence $v_0$ and $v_1$ are both non-zero vectors. Since $U(\F_{q^n})^{\sigma}=U(\F_{q^n})$,  $\la (v_0+v_i)^{\sigma^j}, j=0,1,\ldots, n-1 \ra_{\F_{q^n}}$ is a subspace of $U(\F_{q^n})$, hence $\left\{\displaystyle\sum_{j=0}^{n-1}\lambda^{q^j} (v_0 +v_i)^{\sigma^j}, \lambda \in \F_{q^n}\right\}$ is a subspace of $U$. The projection of $U$ from $\displaystyle\opl_{h >0}V_h$ on $V_0$ contains the $\F_q$-vector space $U'=$ $\{\lambda v_0+ \lambda^{q^{n-i}}v_i^{\sigma^{n-i}}, \lambda \in \F_{q^n}\}$. If $\dim \la v_0,v_i^{\sigma^{n-i}}\ra_{\F_{q^n}} = 2$, then $\dim_{\F_q} U'=n$, hence $U=U'$ and $L_U$ is a linear set of pseudoregulus type. Let $\dim \la v_0,v_i^{\sigma^{n-i}}\ra_{\F_{q^n}} = 1$, hence there exists $\mu \in \F_{q^n}^*$ such that $v_i^{\sigma^{n-i}}=\mu v_0$. Hence the vector space $U'=\{(\lambda+\mu\lambda^{q^{n-i}})v_0, \lambda \in \F_{q^n}\}$ is contained in $U$. Let $P=v_0$. The $\F_q$-linear equation in $\lambda$ $\lambda+\mu\lambda^{q^{n-i}}=0$ has 1 or $q^d$ solutions, where $d=GCD(n,i)$. If $\lambda+\mu\lambda^{q^{n-i}}=0$ only for $\lambda=0$, $w_{L_U}(P)=n$, yielding again a contradiction. Hence, $\lambda+\mu\lambda^{q^{n-i}}=0$ has $q^d$ solutions and $w_{L_U}(P) \geq n-d$. Let  $P=(1,a)$. If $d=1$, then  $w_{L_U}(P)=n-1$, hence $L_U$ is a club and then $\rank (A - diag(a,a^q,\ldots,a^{q^n-1}))=1$. If $d>1$, then $U'$ is an $\F_{q^d}$ vector space of dimension $\frac{n}{d}-1$ and hence  the kernel of $f(x) -ax$ contains the kernel of $Tr_{q^n|q^d}(cx)$ for some $c \in \F_{q^n}^*$. The dimension of the space of the $\F_q$-linear maps with kernel containing a fixed $(n-d)$-subspace is $dn$, hence

\noindent $f(x)-ax \in \{b_0Tr_{q^n|q^d}(cx)+b_1Tr_{q^n|q^d}(cx)^q+\cdots + b_{d-1}Tr_{q^n|q^d}(cx)^{q^{d-1}},b_i \in \F_{q^n}\}$.  In this way, we get $\rank A[\Z_n\setminus \{0,i\}|0,i]=1$. We also must have

\noindent $\rank A[0,i|\Z_n\setminus \{0,i\}]=1$, hence $\det A[0,i|j_1,j_2]=0$ $\forall \, j_1,j_2 \neq 0,i$. Since

 $$A[0,i|j_1,j_2] = \left(
                                                                           \begin{array}{cc}
                                                                             b_{h_1}c^{q^{j_1}} & b_{h_2}c^{q^{j_2}} \\
                                                                             b_{h_1}^{q^i}c^{q^{j_1}} &  b_{h_2}^{q^i}c^{q^{j_2}} \\
                                                                           \end{array}
                                                                         \right)$$

\noindent with $h_i \in \Z_d$ such that $h_i \equiv j_i \mod d$,  $\det A[0,i|j_1,j_2]=0$ if and only if

$$\det \left(
                                                                           \begin{array}{cc}
                                                                             b_{h_1} & b_{h_2} \\
                                                                             b_{h_1}^{q^i} &  b_{h_2}^{q^i} \\
                                                                           \end{array}
                                                                         \right)=0.$$

\noindent As $d=GCD (n,i)$, that implies that $(b_1,\ldots,b_{d-1})=\lambda (b_1',\ldots,b_{d-1}')$, with $b_i' \in \F_{q^d}$ $\forall \, i=1,2,\ldots, d-1$.
\end{proof}

Let $L_U$ a linear set of pseudoregulus type. Then $U=\{\lambda v_0+ \lambda^{q^i} v_1, \lambda \in \F_{q^n}\}$ for some vectors  $v_0,v_1 \in V$ such that $\dim_{\F_{q^n}} \la v_0,v_1\ra=2$ and $1=GCD (n,i)$. Up to the action of $SL(2,q^n)$, we can map $\{v_0,v_1\}$ to $\{\mu_0(1,0), \mu_1(0,1)\}$ for some $\mu_0,\mu_1 \in \F_{q^n}^*$, so that we obtain $\{(\lambda\mu_0,\lambda^{q^i}\mu_1), \lambda \in \F_{q^n}\}$. Finally, by multiplying by $\mu_0^{-1}$, we can assume that $U=\{(x, ax^{q^i}), x \in \F_{q^n}\}$ for some $a \in \F_{q^n}^*$, $i$ such that $GCD (n,i)=1$.
In \cite{CsZ}, it has been proved that for any $W=\{(x,bx^{q^j}),  x \in \F_{q^n}\}$ with $GCD (n,j)=1$ and $N_{q^n|q}(b)=N_{q^n|q}(a)$, $L_W=L_U$. If $j=i$, then $W=\lambda U$ and if $j=n-i$, then $W=\lambda U^{\perp}$. For every $j \neq i,n-1$, $W$ is not equal to $\lambda U$ nor $\lambda U^{\perp}$ for any $\lambda \in F_{q^n}^*$. By \cite[Th.14]{LSZ2015}, we can derive $L_U=L_W$ if and only if $W$ is in the form just described. We include here a proof of the result in our setting.

\begin{theorem}
  Let $U=\{(x,a x^{q^i}), x \in \F_{q^n}\}$,  $a\neq 0$, $GCD(n,i)=1$, and let $W$ be an $n$-dimensional vector space over $\F_q$ such that $L_U=L_W$. Then $W=\{(y,by^{q^j}), y \in \F_{q^n} \}$ with $GCD(n,j)=1$ and $N_{q^n|q}(b)=N_{q^n|q}(a)$.
\end{theorem}

\begin{proof}
  Let $A$ be the matrix associated to $f(x)=ax^{q^i}$, then $A[h|k]\neq 0$ if and only if $k=h+i$. Then $\det A[\alpha|\alpha] \neq 0$ if and only if $\forall \, h \in \alpha$, $h+i \in \alpha$. Since $GCD(n,i)=1$, $\det A[\alpha|\alpha] \neq 0$ if and only if $\alpha=\Z_n$. Let $B$ be a Dickson matrix with the same principal minors as $A$. Then $b_0=0$ and $\det B[0,h|0,h]=b_hb_{n-h}^{q^{h}}=0$ $\forall \, h=1,2,\ldots,n-1$. Let $j=\min\{h \in \{1,2,\ldots,n-1\} \colon b_h \neq 0\}$. If $j=n-1$, then also $B$ has weight row 1. Let $j<n-1$. We prove by induction on $h$ that $b_{n-h}=0$ $\forall \, h=1,2,\ldots,n-j-1$. We have

  $$\det B[0,1,\ldots,j|0,1,\ldots,j]=b_j \det B[1,2,\ldots,j|0,1,\ldots,j-1].$$

\noindent The matrix $B[1,2,\ldots,j|0,1,\ldots,j-1]$ is lower triangular with diagonal $(b_{n-1}^q,b_{n-1}^{q^2},\ldots,b_{n-1}^{q^{j}})$, hence $b_{n-1}=0$. Assume that $b_{n-h'}=0$ $\forall \,h' <h $, then
   $$\det B[0,1,\ldots, j+h-1|0,1,\ldots,j+h-1]=$$

   $$b_j^{1+q+\ldots q^{h-1}}\det B[h,h+1,\ldots,h+j-1|0,1,\ldots,j-1]=$$
   $$b_j^{1+q+\ldots q^{h-1}}b_{n-h}^{q^h+q^{h+1}+\cdots +q^{h+j-1}},$$
hence $b_{n-h}$=0. Therefore, $B$ is a matrix with row weight 1. We observe that $\det A=N_{q^n|q}(a)\det C$ where $C$ is a circulant such $\det C=\displaystyle\prod_{i=0}^{n-1}\omega^i$ and $\omega$ is a primitive $n$th root of unit, hence  $\det A=N_{q^n|q}(a)$. In the same way, we obtain $\det B=N_{q^n|q}(b_j)$. By $\det A = \det B$ we get the statement.
\end{proof}

The two cases left to study are case $(3)$ of Theorem \ref{pseudo} and when $A$ is Loewy-reducible with respect to a partition $\{d\Z_n, \Z_n\setminus d\Z_n\}$ for some $1<d \mid n$. By Proposition \ref{divisor} and Theorem \ref{pseudo}, in both cases we can write $f(x)=f'(x)+\lambda \displaystyle\sum_{i=1}^{d-1}b_iTr_{q^n|q^d}(ax)^{q^i}$, for some $1<d \mid n$, $\lambda, a \in \F_{q^n}^*$, $b_i \in \F_{q^d}$ not all zero, $f'(x)$ $\F_{q^d}$-linear.
By replacing $U=\{(x,f(x)), x \in \F_{q^n}\}$ by $\lambda^{-1}U$, we can assume  $f(x)=f'(x)+\displaystyle\sum_{i=1}^{d-1}b_iTr_{q^n|q^d}(ax)^{q^i}$. The space $U$ contains the $(n-d)$-dimensional subspace

\noindent $U_d:=\{(x,f'(x)), x \in \F_{q^n} \colon Tr_{q^n|q^d}(ax)=0\}$, which is in fact an $\F_{q^d}$-vector space, hence every point of $L_{U_d}$ has weight multiple of $d$. Let $\xi \in \F_{q^n}$ be such that $Tr_{q^n|q^d}(a\xi)=1$,  then $\F_q$-subspace of $U$ $U_{\xi}:=\{(\xi x,f(\xi)x+\displaystyle\sum_{i=1}^{d-1}b_ix^{q^i}), x \in \F_{q^d}\}$ is contained in $\{(\xi x,f(\xi)x+y), x,y \in \F_{q^d}\} =\langle (\xi,f(\xi)),(0,1)\rangle_{\F_{q^d}}$ which, in turn, induces a subline $\ell_{\xi} \cong \PG(1,q^d)$. Clearly, $U=U_d\oplus U_{\xi}$. Let $u \in U_{\xi}\setminus \{\mathbf{0}\}$, then $U'=\la U_d, u \ra_{\F_{q^d}}$ is an $\F_{q^d}$-vector space of dimension $\frac{n}{d}$ and $L_{U'}$ is an $\F_{q^d}$-linear set of $\PG(1,q^n)$ maximum rank. Since $\dim_{\F_q} U'\cap U \geq n-d+1$, $L_{U'}\subset L_U$. Also, let $u_1, u_2 \in U_{\xi}\setminus \{\mathbf{0}\}$, $U_i:=\la U_d, u_i\ra_{\F_{q^d}}$, $i=1,2$, then, by Theorem \ref{hyperplane}, $L_{U_1}=L_{U_2}$ if and only if $U_1=U_2$. By $U_d\cap U_{\xi}=\mathbf{0}$, $U_1=U_2$ if and only if $u_2= \mu u_1$ for some $\mu \in \F_{q^d}^*$. Therefore,

$$
L_U=\displaystyle\bigcup_{u_i \in U_{\xi}\setminus \{\mathbf{0}\}} L_{\langle U_d, u_i \rangle_{\F_{q^d}}}
$$

\noindent and $L_{\langle U_d, u_1 \rangle_{\F_{q^d}}} = L_{\langle U_d, u_2 \rangle_{\F_{q^d}}}$ if and only if $u_1$ and $u_2$ determine the same point of the subline $\ell_{\xi}$.

Let $W_\xi = \{(\xi y, f(\xi)y +g(x)), y \in \F_{q^d}\}$ be such that $L_{W_{\xi}} = L_{U_{\xi}}$ and let $W = U_d \oplus W_{\xi}$. Then, $L_U = L_W$. If $W_{\xi} = \mu U_{\xi}$ for some $\mu \in \F_{q^d}$, then $W = \mu U$, as $\mu U_d = U_d$. If $W_{\xi} = \mu U_{\xi}^{\perp_d}$ for some $\mu \in \F_{q^d}$, then $W$ is the generalized perp of $U$ (see Introduction) and $W \neq \lambda U^{\perp}$ unless $\{(x,f(x)'), x \in \F_{q^n}\} = \lambda \{(x,f(x)'), x \in \F_{q^n}\}^{\perp}$ for some $\lambda \in \F_{q^n}^*$. If $L_{W_{\xi}} = L_{U_{\xi}}$ is a linear set of pseudoregulus type and $W_{\xi} \neq \mu U_{\xi}$, $W_{\xi} \neq \mu U_{\xi}^{\perp_d}$ for any $\mu \in \F_{q^d}$, then $W \neq \lambda U$ and $W \neq \lambda U^{\perp}$ for any $\lambda \in \F_{q^n}^*$. Suppose that $\displaystyle\sum_{i=1}^{d-1}b_iy^{q^i}=f''(y)+\displaystyle\sum_{j=1}^{e-1}b'_jTr_{q^d|q^e}(a'y)^{q^i}$ for some $ 1< e \mid d$, $f''(y)$ $\F_{q^e}$-linear, $b_j' \in \F_{q^e}$, $j=1,2,\ldots,e-1$, and $a' \in \F_{q^d}^*$, then

\noindent $f'(x)+\displaystyle\sum_{i=1}^{d-1}b_iTr_{q^n|q^d}(ax)^{q^i}= f'(x)+f''(Tr_{q^n|q^d}(ax))+\displaystyle\sum_{j=1}^{e-1}b_j' Tr_{q^d|q^e}(a'Tr_{q^n|q^d}(ax))^{q^j}=$

\noindent $f'''(x)+\displaystyle\sum_{j=1}^{e-1}b_j' Tr_{q^n|q^e}(a'ax)^{q^j}$, with $f'''(x)$ $\F_{q^e}$-linear, hence we should look on the reducibility of the matrix associated to $\displaystyle\sum_{j=1}^{e-1}b_j' y^{q^j}$ and we can reduce to one of the cases above described. In the last Theorem of this Section, we will prove that $L_U = L_W$ if and only if $W$ is in the form just described.

\begin{theorem}\label{last}
Let $U=\{(x,f'(x)+\lambda \displaystyle\sum_{i=1}^{d-1}b_iTr_{q^n|q^d}(ax)^{q^i}), x \in \F_{q^n}\}$, for some $1<d \mid n$, $\lambda, a \in \F_{q^n}^*$, $b_i \in \F_{q^d}$ not all zero, $f'(x)$ $\F_{q^d}$-linear. Then $L_U = L_W$ if and only if $W = U_d \oplus W_{\xi}$ and $L_{W_{\xi}} = L_{U_{\xi}}$ in $\ell_{\xi}$.
\end{theorem}

\begin{proof}
Let $W=\{(x,g(x)), x \in \F_{q^n}\}$ and let $B$ be the matrix associated to $g(x)$. Let $g(x)=g'(x)+h(x)$, where $g'$ is the linear function associated to $B[d\Z_n|d\Z_n]$, hence $g'(x)$ is $\F_{q^d}$-linear. By Theorem \ref{dlinear}, $W$ contains an $\F_{q^d}$-vectors space of dimension $\frac{n}{d}-1$, say $W_d$. Let $h(x)=\displaystyle \sum_{i=1}^{d-1}h_i(x)^{q^i}$, where $h_i(x)$ is $\F_{q^d}$-linear, for $i=1,2,\ldots,d-1$. We have $g(\lambda x)= g'(\lambda x)+\displaystyle \sum_{i=1}^{d-1}\lambda^{q^i}h_i(x)^{q^i}=\lambda g(x)=\lambda g'(x)+\lambda \displaystyle \sum_{i=1}^{d-1}h_i(x)^{q^i}$ for every $\lambda \in \F_{q^d}$, $x \in W_d$, hence $h_i(x)=0$ for every $x \in W_d$,  $i=1,2,\ldots,d-1$. Hence we get $g(x)=g'(x)+\displaystyle\sum_{i=1}^{d-1}c_iTr_{q^n|q^d}(a'x)^{q^i}$, with $a'\neq 0$. Since $A[d\Z_n|d\Z_n]$ and $B[d\Z_n|d\Z_n]$ have the same principal minors and $L_{U_d}=L_{W_d}$, by Theorem \ref{hyperplane},
\medskip

\noindent $B[d\Z_n|d\Z_n]= diag(\frac{a'}{a},(\frac{a'}{a})^q,\ldots,(\frac{a'}{a})^{q^{n-1}})^{-1} \, A[d\Z_n|d\Z_n] \, diag(\frac{a'}{a},(\frac{a'}{a})^q,\ldots,(\frac{a'}{a})^{q^{n-1}})$, hence, by substituting $W$ by $\frac{a'}{a}W$ and renaming the $c_i'$s, we get $g(x) = f(x)'+$ $\displaystyle\sum_{i=1}^{d-1}c_iTr_{q^n|q^d}(ax)^{q^i}$. Let $A'$ be the $d\times d$ matrix such that $A'[i|j]=b_{j-i}^{q^i}$, $i,j=0,1,\ldots,d-1$, and $B'$ be such that $B'[i|j]=c_{j-i}^{q^i}$, $i,j=0,1,\ldots,d-1$. The matrices

\noindent$A[0,1,\ldots,d-1|0,1,\ldots,d-1]$ and $B[0,1,\ldots,d-1|0,1,\ldots,d-1]$ have equal corresponding principal minors, hence $\displaystyle\prod_{i \in I}a^{q^i}\det A'[I|I]= \displaystyle\prod_{i \in I}a^{q^i}\det B'[I|I]$ for every subset $I$ of $\{0,1,\ldots,d-1\}$. Since $a\neq 0$, $A'$ and $B'$ have equal corresponding principal minors too. Hence

\noindent $\{(x,\displaystyle\sum_{i=1}^{d-1}b_ix^{q^i}), x \in \F_{q^d}\}$ and $\{(y,\displaystyle\sum_{i=1}^{d-1}c_iy^{q^i}), y \in \F_{q^d}\}$ determine the same set of points. Since $b_i \in \F_{q^d}$ $\forall \, i=1,2,\ldots,d-1$, we get that $\displaystyle\sum_{i=1}^{d-1}c_iy^{q^i-1} \in \F_{q^d}$ for every $y \in \F_{q^d}$, and hence $c_i \in \F_{q^d}$ $\forall \, i=1,2,\ldots,d-1$. By $L_{U_{\xi}}=L_{W_{\xi}}$, we get $L_U=L_W$.

\end{proof}

The construction of generalized perp or generalized pseudoregulus linear sets can be easily generalized for linear sets of different ranks in $\PG(r-1,q^n)$. Let $n$ not a prime and $1< d \mid n$.  Let $U_d$ be an $\F_{q^d}$ vector space of dimension $m \leq (r-1)\frac{n}{d}-1$ and let $U_1,W_1$ $\F_q$ vector spaces of dimension $d$ disjoint from $U_d$ such that they determine the same linear set $L_{U_1} = L_{W_1}$ of a subline $\ell \cong \PG(1,q^d)$ with $W_1 = \mu U_1^{\perp_d}$ for some $\mu \in \F_{q^d}^*$ or  $L_{U_1} = L_{W_1}$ is a linear set of psudoregulus type. Then $U:= U_d \oplus U_1$ and $W:= U_d \oplus W_1$ are $\F_q$ vector spaces of dimension $(m+1)d$ and $L_U = L_W$ with $W$ not necessarily equal to $\lambda U$ for some $\lambda \in \F_{q^n}^*$. We stress  that for $r > 2$, we can take also $d=n$. In fact, in \cite{marino},  it was already pointed out that a linear set of $\PG(r-1,q^n)$, $r > 2$  that is a cone with vertex a projective subspace (hence $\F_{q^n}$-linear) and base a linear set of pseudoregulus type, can be determined by vector spaces which are not one multiple of the other.

We conclude with a characterization of linear sets of maximum rank of $\PG(r-1,q^n)$ $U$ and $W$ such that $L_U=L_W$ and there exists at least a line intersecting $L_U=L_W$ in a linear set of pseudoregolus type.

\begin{theorem}
Let $U$ and $W$ be $\F_q$-vector spaces such that $\dim_{\F_q} U=\dim_{\F_q}W=(r-1)n$ and $L_U = L_W$ in $\PG(r-1,q^n)$. Let $U_0$ and $W_0$ be the intersection of $U$ and $W$ respectively with a line $\ell$ of $\PG(r-1,q^n)$, such that $\dim_{\F_q} U_0 = \dim_{\F_q} W_0 =n$ and$L_{U_0} = L_{W_0}$ is a linear set of pseudoregulus type with maximum field of linearity $\F_q$. Then either $W=\lambda U$ for some $\lambda \in \F_{q^n}^*$, or $L_U=L_W$ is cone with vertex a codimension 2 subspace of $\PG(r-1,q^n)$ disjoint from $\ell$ and base $L_{U_0}=L_{W_0}$.
\end{theorem}
\begin{proof}
We prove the result for $r=3$, then it will easily follow for any $r \geq 3$.
Let $U=\{(x,y,f_0(x)+f_1(y)), x, y \in \F_{q^n}\}$, $W=\{(x,y,g_0(x)+g_1(y)), x, y \in \F_{q^n}\}$.
By the action of PGL$(3,q^n)$, we can assume that $f_0(x)=ax^{q^i}$. Then $g_0(x)=bx^{q^j}$ with $GCD(j,n)=GCD(i,n)=1$, $N_{q^n|q}(a) = N_{q^n|q}(b)$.  Let $A$ and $B$ be the matrices associated to $f_0$ and $g_0$ respectively, $A'$ such that $A'[h||0]=a_h$ $\forall \, h \in \Z_n$, $A'[h|k]=A[h|k]$ for every $k>0$, $h \in \Z_n$, $B'[h|0]=b_h$ $\forall \, h \in \Z_n$, $B'[h|k]=B[h|k]$ for every $k>0$, $\forall \, h \in \Z_n$. Suppose that $A'$ and $B'$ have equal corresponding principal minors. Since the columns of $A'$ distinct from the 0-th one have weight 1, we have that $\det A'[\{0\}\cup \alpha |\{0\}\cup \alpha ]$ is not identically 0 only if $\forall \, h \in \alpha, h+i \in \{0\}\cup \alpha$, hence $\alpha=\{n-i,n-2i,\ldots,n-di\}$ and $\det A'[\{0\}\cup \alpha |\{0\}\cup \alpha ]=a_{n-di}$. Since the subgroup of $(\Z_n,+)$ generated by $i$ is $\mathbb{Z}_n$, if $i=j$, then $a_h=b_h$ $\forall h=1,2,\ldots,n-1$.
  If $i \neq j$, then $\{i,2i,\ldots,di\} \neq \{j,2j,\ldots, dj\}$ unless $d=n$, hence $a_h=b_h=0$ $\forall h=1,2,\ldots,n-1$. Hence, by Theorem \ref{main2}, $f_1(y)=g_1(y)=cy$, so $L_U=L_W$ is a cone with vertex $(0,1,c)$ and base a linear set of pseudoregulus type.
\end{proof}


\bigskip
	\noindent
	Valentina Pepe\\
	Dipartimento di Scienze di Base ed Applicate per l' Ingegneria,\\
	``Sapienza" Universit\`{a}
	di Roma, \\
	Via Antonio Scarpa, 10, 00161 Roma, Italy\\
	{\em valentina.pepe@uniroma1.it}

\end{document}